\documentclass{amsart}

\usepackage[latin1]{inputenc}
\usepackage[T1]{fontenc}
\usepackage{enumerate}

\newtheorem{theorem}{Theorem}[section]
\newtheorem{lemma}[theorem]{Lemma}
\newtheorem{cor}[theorem]{Corollary}
\newtheorem{prop}[theorem]{Proposition}

\theoremstyle{definition}
\newtheorem{definition}[theorem]{Definition}
\newtheorem{example}[theorem]{Example}

\theoremstyle{remark}
\newtheorem{remark}[theorem]{Remark}

\numberwithin{equation}{section}

\newcommand{\N}{\mathbb N}
\newcommand{\K}{\mathbb K}
\newcommand{\R}{\mathbb R}
\newcommand{\Z}{\mathbb Z}
\newcommand{\C}{\mathbb C}

\allowdisplaybreaks[3]
\begin{document}

\title{Hypercyclic subspaces on Fréchet spaces without continuous norm}
\author[Q. Menet]{Quentin Menet}
\address{Institut de Mathématique\\
Université de Mons\\
20 Place du Parc\\
7000 Mons, Belgique}
\email{Quentin.Menet@umons.ac.be}
\thanks{The author is supported by a grant of FRIA}


\date{}

\begin{abstract}
Known results about hypercyclic subspaces concern either Fréchet spaces with a continuous norm or the space $\omega$. We fill the gap between these spaces by investigating Fréchet spaces without continuous norm. To this end, we divide hypercyclic subspaces into two types: the hypercyclic subspaces $M$ for which there exists a continuous seminorm $p$ such that $M\cap \ker p=\{0\}$ and the others.  For each of these types of hypercyclic subspaces, we establish some criteria. This investigation permits us to generalize several results about hypercyclic subspaces on Fréchet spaces with a continuous norm and about hypercyclic subspaces on $\omega$. In particular, we show that each infinite-dimensional separable Fréchet space supports a mixing operator with a hypercyclic subspace.
\end{abstract}
\maketitle

\section*{Introduction}

We denote by $\Z$ the set of integers, by $\N$ the set of positive integers and by $\Z_+$ the set of non-negative integers.

Let $(T_k)$ be a sequence of linear continuous operators from $X$ to $Y$ where $X$ is an infinite-dimensional Fréchet space and $Y$ is a separable topological vector space.
The sequence $(T_k)$ is said to be hypercyclic if there exists a vector $x$ in $X$ (also called hypercyclic) such that the orbit of $x$ for $(T_k)$ is dense in $Y$. We are interested in the existence of closed infinite-dimensional subspaces in which every non-zero vector is hypercyclic. Such a subspace is called a hypercyclic subspace.

Some classical hypercyclic operators, like the translation operators on the space of entire functions, possess a hypercyclic subspace \cite{Bernal0} but some others, like scalar multiples of the backward shift on $l^p$, do not possess any hypercyclic subspace \cite{Montes}. A natural question is thus: "Which hypercyclic operators possess a hypercyclic subspace?". Several criteria have been found for operators on Fréchet spaces with a continuous norm (see \cite{Bernal2}, \cite{Bonet}, \cite{Leon}, \cite{Menet2}, \cite{Montes} and \cite{Petersson}) and a characterization has even been given for weakly mixing operators on complex Banach spaces by Gonz\'alez, Le\'on and Montes~\cite{Gonzalez}. Moreover, B\`{e}s and Conejero \cite{Bes} have shown that there exist some hypercyclic operators with hypercyclic subspaces on the space $\omega$, where $\omega $ is the space $\K^{\Z_+}$ ($\K=\R$ or $\C$) endowed with the product topology. On the other hand, Charpentier, Mouze and the author \cite{Charpentier2} have obtained the existence of hypercyclic subspaces for universal series on some Fréchet spaces without continuous norm. Nevertheless, no criterion is known in the case of Fréchet spaces without continuous norm.

The starting point of this paper is the following simple remark. There exist two types of hypercyclic subspace $M$ for $(T_k)$: either there exists a continuous seminorm $p$ on $X$ such that $M\cap \ker p=\{0\}$ (type $1$), or for any continuous seminorm $p$ on $X$, the subspace $M\cap \ker p$ is infinite-dimensional (type $2$).
If $X$ possesses a continuous norm, then each hypercyclic subspace is of type $1$. On the other hand, if each continuous seminorm of $X$ has a kernel of finite codimension, then each hypercyclic subspace is of type $2$; this is the case of $\omega$, for example.
However, in the other cases, there can exist hypercyclic subspaces of type $1$ or of type $2$.

In Section~\ref{type1}, we consider hypercyclic subspaces of type $1$. We start by proving the existence of convenient basic sequences in each Fréchet space admitting a continuous seminorm whose kernel is not of finite codimension. Thanks to these basic sequences, we generalize principal criteria about hypercyclic subspaces on Fréchet spaces with a continuous norm. An interesting application of these results concerns the differential operators on the space $C^{\infty}(\R)$. We prove in Section~\ref{cinfty} that for any non-constant polynomial $P$, the operator $P(D)$ possesses a hypercyclic subspace in $C^{\infty}(\R)$. We can also prove that, on each Fréchet sequence space admitting a continuous seminorm whose kernel is not of finite codimension, the existence of one restricted universal series implies the existence of a closed infinite-dimensional subspace of restricted universal series (Section \ref{section univ}).
Finally, we answer positively a question posed by  B\`{e}s and Conejero in \cite[Problem 8]{Bes}: "Does every separable infinite-dimensional Fréchet space support an operator with a hypercyclic subspace?" (Section \ref{existence}).

In Section~\ref{type2}, we focus on hypercyclic subspaces of type $2$. We establish a sufficient criterion for having a hypercyclic subspace of type $2$ and a sufficient criterion for having no hypercyclic subspace of type $2$. These criteria are applied in Section~\ref{extype2} to three classes of hypercyclic operators. We first look at the case of universal series on Fréchet sequence spaces for which each continuous seminorm has a kernel of finite codimension (Section \ref{univ series}). In particular, we generalize to these spaces a result obtained for $\omega$ in \cite{Charpentier2}. This result together with results obtained in Section~\ref{section univ} characterizes almost completely the existence of closed infinite-dimensional subspaces of restricted universal series in Fréchet spaces.  Afterwards, we improve a result obtained by B\`{e}s and Conejero \cite{Bes} that states that the operators of the form $P(B_w)$, where $P$ is a non-constant polynomial and $B_w$ is a weighted shift, possess a hypercyclic subspace on $\omega$; we show that a larger class of sequences of operators from $\omega$ to $\omega$ possesses a hypercyclic subspace and even a frequently hypercyclic subspace (Section \ref{KN}). Finally, we investigate unilateral weighted shifts on Fréchet sequence spaces without continuous norm (Section \ref{weighted}).

\section{Some criteria for hypercyclic subspaces of type $1$}\label{type1}
We start by proving the simple remark stated in the Introduction.

\begin{prop}\label{prop2type}
Let $X$ be a Fréchet space, $(p_n)_{n\ge 1}$ an increasing sequence of seminorms inducing the topology of $X$ and $M$ an infinite-dimensional subspace in $X$. Either for any ${n\ge 1}$, $\ker p_n\cap M$ is infinite-dimensional or there exists $n\ge 1$ such that $\ker p_n\cap M=\{0\}$.
\end{prop}
\begin{proof}
Suppose that there exists $n\ge 1$ such that $\ker p_n\cap M$ is finite-dimensional.
Let $e_1,\dots,e_d$ be a basis of $\ker p_n\cap M$. Since $X$ is a Fréchet space, there exists $m\ge n$ such that
$p_m(e_1)\ne 0$. Therefore, the dimension of $\ker p_m\cap M$ is strictly less than $d$. By repeating this argument, we then obtain a seminorm $p_N$ such that $\ker p_N\cap M=\{0\}$.
\end{proof}

In general, the simplest way to obtain a hypercyclic subspace is to construct a convenient basic sequence in $X$ and to consider the closed linear span of this sequence.

\begin{definition}
A sequence $(u_n)_{n\ge 1}$ in a Fréchet space is called \emph{basic} if for every $x\in \overline{\text{span}}\{u_n:n\ge 1\}$, there exists a unique sequence $(a_n)_{n\ge1}$ in $\mathbb{K}$ such that $x=\sum_{n=1}^{\infty}a_nu_n$.
\end{definition}

Let $X$ be a Fréchet space. If $X$ possesses a continuous norm, the existence of basic sequences in $X$ is well known (see \cite{Menet}, \cite{Petersson}).
We now show that if there exists a continuous seminorm $p$ on $X$ such that $\ker p$ is not a subspace of finite codimension, we can generalize the classical construction of basic sequences in Fréchet spaces with a continuous norm to obtain basic sequences $(u_k)\subset X$ such that $p(u_k)=1$ for any $k$ and $\overline{\text{span}}\{u_k:k\ge 1\}\cap \ker p=\{0\}$. Moreover, sufficiently small perturbations of these sequences will remain basic and will be equivalent to the initial sequence.

\begin{definition}
Let $X$ be a Fréchet space.
Two basic sequences $(u_n)$ and $(f_n)$ in $X$ are said to be \emph{equivalent} if for every sequence $(a_n)_{n\ge 1}$ in $\mathbb{K}$, the series $\sum_{n=1}^{\infty} a_n u_n$ converges in $X$ if and only if $\sum_{n=1}^{\infty} a_n f_n$ converges in $X$.
\end{definition}

These sequences will be the key point to obtain criteria about hypercyclic subspaces of type $1$. 

\begin{lemma}[{\cite[Lemma 10.39]{Karl}}]\label{lem bas}
Let $X$ be a Fréchet space. For any finite-dimensional subspace $F$ of $X$, for any continuous seminorm $p$ on $X$ and for any $\varepsilon>0$, there exists a closed subspace $E$ of finite codimension such that for any $x\in E$, for any $y\in F$, we have
\[p(x+y)\ge \max\Big(\frac{p(x)}{2+\varepsilon},\frac{p(y)}{1+\varepsilon}\Big).\]
\end{lemma}

\begin{cor}\label{corb}
Let $X$ be a Fréchet space, $(p_n)$ a sequence of continuous seminorms and $M$ an infinite-dimensional subspace such that for any closed subspace $E$ of finite codimension, we have \[E\cap M\not\subset \ker p_1.\]
Then for any $\varepsilon>0$, for any $u_1,\dots,u_n\in X$, there exists $u_{n+1}\in M$ such that $p_1(u_{n+1})=1$ and such that for any $j\le n$, for any  $a_1,\ldots,a_{n+1}\in\K$, we have
\[p_j\Big(\sum_{k=1}^n a_k u_k\Big)\le (1+\varepsilon)p_j\Big(\sum_{k=1}^{n+1}a_k u_k\Big).\]
\end{cor}
\begin{remark}
If $\ker p_1$ is a subspace of finite codimension, such a subspace $M$ does not exist. However, if $\ker p_1$ is not a subspace of finite codimension, then $M=X$ works and if $\ker p_1=\{0\}$, then each infinite-dimensional subspace $M$ works. Moreover, we notice that if $M\cap \ker p_1=\{0\}$, then $M$ works.
\end{remark}

\begin{theorem}\label{thm bas}
Let $X$ be a Fréchet space, $(p_n)$ an increasing sequence of seminorms defining the topology of $X$, $(\varepsilon_n)_{n\ge 1}$ a sequence of positive real numbers such that $\prod_{n} (1+\varepsilon_n)=K<\infty$. If a sequence $(u_k)_{k\ge 1}\subset X$ satisfies for any $n\in \N$, for any $j\le n$, for any $a_1,\ldots,a_{n+1}\in \K$,
\begin{equation}
p_j\Big(\sum_{k=1}^n a_k u_k\Big)\le (1+\varepsilon_n)p_j\Big(\sum_{k=1}^{n+1}a_k u_k\Big)\quad \text{and}\quad p_1(u_n)=1,\label{lem 1}\end{equation}
then this sequence is basic in $X$ and the closure of the linear span $M_u$ of $(u_k)$ satisfies $M_u\cap \ker p_1=\{0\}$.
Moreover, if $(f_k)_{k\ge 1}\subset X$ satisfies
\[\sum_{k=1}^{+\infty}2Kp_k(u_k-f_k)<1,\]
then $(f_k)$ is a basic sequence in $X$, $(f_k)$ is equivalent to $(u_k)$ and the closure of the linear span $M_f$ of $(f_k)$ satisfies $M_f\cap \ker p_1=\{0\}$.
\end{theorem}
\begin{proof}
We first show that $(u_k)_{k\ge 1}$ is a basic sequence. Let $x=\sum_{k=1}^{\infty}a_ku_k\in X$.
We remark that, by \eqref{lem 1}, we have for any $n\ge1$
\begin{equation}
\begin{aligned}
|a_n|=p_1(a_n u_n)&\le p_1\Big(\sum_{k=1}^na_k u_k\Big)+p_1\Big(\sum_{k=1}^{n-1}a_ku_k\Big)\\
&\le 2Kp_1\Big(\sum_{k=1}^{\infty}a_k u_k\Big)=2Kp_1(x).
\end{aligned}
\label{eq4}
\end{equation}
This property implies that if  $x=\sum_{k=1}^{\infty}a_ku_k$ and $x=\sum_{k=1}^{\infty}b_ku_k$ then for any $n\ge 1$, $a_n=b_n$.
Let $x_n=\sum_{k=1}^{\infty} a_{n,k} u_k$, with $a_{n,k}=0$ for any $k\ge N_n$, be a convergent sequence to some vector $x\in X$. It just remains to prove that there exists a sequence $(a_k)_{k\ge 1}\subset \K$ such that $x=\sum_{k=1}^{\infty}a_ku_k$. 

We deduce from \eqref{eq4} that for any $k\ge 1$, $|a_{m,k}-a_{n,k}|\le 2K p_1(x_m-x_n)\rightarrow 0$ as $m,n\rightarrow \infty$ and therefore $\lim_{n\rightarrow\infty}a_{n,k}=a_k$ for some scalar $a_k$.
Let $j,n\ge 1$. For any $N\ge j$, we have by \eqref{lem 1}
\begin{align*}
p_j\Big(\sum_{k=1}^Na_{n,k}u_k-\sum_{k=1}^Na_ku_k\Big)&=\lim_m p_j\Big(\sum_{k=1}^Na_{n,k}u_k-\sum_{k=1}^Na_{m,k}u_k\Big)\\
&\le \lim_m Kp_j(x_n-x_m)= K p_j(x_n-x).\end{align*}
Thus, for any $N\ge \max\{N_n,j\}$, we have
\[p_j\Big(x-\sum_{k=1}^{N}a_ku_k\Big)\le p_j(x-x_n)+p_j\Big(\sum_{k=1}^Na_{n,k}u_k-\sum_{k=1}^Na_ku_k\Big)
\le (1+K) p_j(x_n-x).\]
We deduce that $x=\sum_{k=1}^{\infty}a_k u_k$ and thus the sequence $(u_k)$ is basic. Moreover, if $M_u=\overline{\text{span}}\{u_k:k\ge 1\}$, then for any $x\in M_u$, $x= \sum_{k=1}^{\infty}a_ku_k$ with $|a_k|\le 2Kp_1(x)$. We conclude that $M_u\cap \ker p_1=\{0\}$.\\

We now consider a sequence $(f_k)_{k\ge 1}\subset X$ such that
\begin{equation}\delta:=\sum_{k=1}^{+\infty}2Kp_k(u_k-f_k)<1.\label{eq5}\end{equation}
In order to prove that $(f_k)$ is a basic sequence, we consider the operator $T:M_u\rightarrow X$ given by $T\big(\sum_{n=1}^{\infty}a_n u_n\big)= \sum_{n=1}^{\infty}a_n f_n$. This operator is well-defined as for any $j\ge 1$, for any $j\le m\le n$,
\begin{align}
p_j\Big(\sum_{k=m}^{n}a_k f_k\Big)&\le p_j\Big(\sum_{k=m}^{n}a_k (f_k-u_k)\Big)+ p_j\Big(\sum_{k=m}^{n}a_k u_k\Big)\nonumber\\
&\le \sum_{k=m}^n |a_k| p_j(f_k-u_k) + p_j\Big(\sum_{k=m}^{n}a_k u_k\Big)\nonumber\\
&\le \sum_{k=m}^n 2Kp_1\Big(\sum_{i=m}^{n}a_i u_i\Big) p_k(f_k-u_k) + p_j\Big(\sum_{k=m}^{n}a_k u_k\Big) \quad\text{by \eqref{eq4}}\nonumber\\
&\le (1+\delta)p_j\Big(\sum_{k=m}^{n}a_k u_k\Big) \quad\text{by \eqref{eq5}}
\label{eq6}
\end{align}
and the operator $T$ is continuous as, for any $j\ge 1$, we have with the same reasoning:
\begin{align*}
p_j\Big(T\Big(\sum_{k=1}^{\infty}a_k u_k\Big)\Big)&\le 
\sum_{k=1}^\infty 2Kp_1\Big(\sum_{i=1}^{\infty}a_i u_i\Big) p_j(f_k-u_k) + p_j\Big(\sum_{k=1}^{\infty}a_k u_k\Big)\\
&\le \sum_{k=1}^{j-1}2Kp_1\Big(\sum_{i=1}^{\infty}a_i u_i\Big) p_j(f_k-u_k)+(1+\delta) p_j\Big(\sum_{k=1}^{\infty}a_k u_k\Big)\\
&\le \Big(1+\delta+\sum_{k=1}^{j-1}2K p_j(f_k-u_k)\Big) p_j\Big(\sum_{k=1}^{\infty}a_k u_k\Big).
\end{align*}

We seek to prove that $T$ is an isomorphism between $M_u$ and $\text{Im}(T)$.
We remark that for any $\sum_{k=1}^{\infty}a_k u_k\in M_u$, for any $j\ge 1$, we have by \eqref{eq4} and \eqref{eq5}
\begin{align}
p_j\Big(\sum_{k=j}^{\infty}a_k f_k\Big)
&\ge p_j\Big(\sum_{k=j}^{\infty}a_k u_k\Big)-p_j\Big(\sum_{k=j}^{\infty}a_k (f_k-u_k)\Big)\nonumber\\
&\ge p_j\Big(\sum_{k=j}^{\infty}a_k u_k\Big)- \sum_{k=j}^{\infty} 2Kp_j\Big(\sum_{i=j}^{\infty}a_i u_i\Big) p_k(f_k-u_k)\nonumber\\
&\ge  p_j\Big(\sum_{k=j}^{\infty}a_k u_k\Big)- \delta p_j\Big(\sum_{k=j}^{\infty}a_k u_k\Big)\nonumber\\
&\ge (1-\delta) p_j\Big(\sum_{k=j}^{\infty}a_k u_k\Big).\label{eq7}
\end{align}
Thanks to \eqref{eq7} with $j=1$, we can already assert that $T$ is injective, as for any $x\in M_u\backslash\{0\}$, we have $p_1(x)\ne0$. We also deduce that $\text{Im}(T)\cap \ker p_1=\{0\}$.

Let $x_n=\sum_{k=1}^{\infty}a_{n,k}u_k$ be a sequence in $M_u$ such that $Tx_n$ converges to $f$ in $X$ as $n\rightarrow \infty$. To prove that $T$ is an isomorphism, we still have to show that $x_n$ converges in $X$. Thanks to \eqref{eq4} and \eqref{eq7}, we know that for any $k\ge 1$, $a_{n,k}\rightarrow a_k$ for some scalar $a_k$ as $n\rightarrow \infty$. We therefore seek to prove that $x_n$ converges to $\sum_{k=1}^\infty a_k u_k$. To this end, we begin by showing that $\sum_{k=1}^\infty a_k f_k$ converges.
For any $N\ge M\ge j$, any $n\ge 1$, we have
\begin{align*}
&p_j\Big(\sum_{k=M}^N a_k f_k\Big)\\
&\quad= \lim_m p_j\Big(\sum_{k=M}^N a_{m,k} f_k\Big)\\
&\quad\le \lim_m p_j\Big(\sum_{k=M}^N a_{m,k} f_k-\sum_{k=M}^N a_{n,k} f_k\Big)+p_j\Big(\sum_{k=M}^N a_{n,k} f_k\Big)\\
&\quad\le (1+\delta)\lim_m p_j\Big(\sum_{k=M}^N (a_{m,k}- a_{n,k}) u_k\Big)+p_j\Big(\sum_{k=M}^N a_{n,k} f_k\Big) \quad\text{by \eqref{eq6}}\\
&\quad\le 2K(1+\delta)\limsup_m p_j\Big(\sum_{k=j}^{\infty} (a_{m,k}-a_{n,k}) u_k\Big)+p_j\Big(\sum_{k=M}^N a_{n,k} f_k\Big)\quad\text{by \eqref{lem 1}}\\
&\quad\le \frac{2K(1+\delta)}{1-\delta}\limsup_m p_j\Big(\sum_{k=j}^{\infty} (a_{m,k}-a_{n,k}) f_k\Big) +p_j\Big(\sum_{k=M}^N a_{n,k} f_k\Big)\quad\text{by \eqref{eq7}}\\
&\quad\le \frac{2K(1+\delta)}{1-\delta}\limsup_m p_j(Tx_m-Tx_n)\\
&\quad\quad+\frac{2K(1+\delta)}{1-\delta}\limsup_m p_j\Big(\sum_{k=1}^{j-1} a_{m,k} f_k-\sum_{k=1}^{j-1} a_{n,k} f_k\Big)
+p_j\Big(\sum_{k=M}^N a_{n,k} f_k\Big)\\
&\quad\le \frac{2K(1+\delta)}{1-\delta} p_j(f-Tx_n)\\
&\quad\quad+ \frac{2K(1+\delta)}{1-\delta} p_j\Big(\sum_{k=1}^{j-1} a_{k} f_k-\sum_{k=1}^{j-1} a_{n,k} f_k\Big)
+p_j\Big(\sum_{k=M}^N a_{n,k} f_k\Big).
\end{align*}
If we choose $n\ge 1$ such that 
\[\frac{2K(1+\delta)}{1-\delta} p_j(f-Tx_n)
+ \frac{2K(1+\delta)}{1-\delta}p_j\Big(\sum_{k=1}^{j-1} a_{k} f_k-\sum_{k=1}^{j-1} a_{n,k} f_k\Big)<\varepsilon\]
and $L\ge j$ such that, for any $N\ge M\ge L$, we have
\[p_j\Big(\sum_{k=M}^N a_{n,k} f_k\Big)<\varepsilon,\]
then we deduce that, for any $N\ge M\ge L$, we have
\[p_j\Big(\sum_{k=M}^N a_k f_k\Big)\le 2\varepsilon.\]
The sequence $(\sum_{k=1}^{N}a_kf_k)_N$ is thus a Cauchy sequence and therefore a convergent sequence.
Moreover, as for any $N\ge M\ge j$, we have by \eqref{eq7}
\[p_j\Big(\sum_{k=M}^Na_ku_k\Big)\le \frac{1}{1-\delta}p_j\Big(\sum_{k=M}^Na_kf_k\Big),\]
we deduce that the sequence $(\sum_{k=1}^{N}a_ku_k)$ is also a Cauchy sequence and therefore $\sum_{k=1}^{\infty}a_ku_k$ converges.

We can now show that $f=\sum_{k=1}^{\infty} a_k f_k$. Indeed, since $\sum_{k=1}^{\infty} a_k f_k$ and $\sum_{k=1}^{\infty} a_{n,k} f_k$ are convergent series, for any $n,j\ge 1$, there exists $M\ge j$ such that we have
\[p_j\Big(\sum_{k=M+1}^{\infty}a_kf_k-\sum_{k=M+1}^{\infty}a_{n,k}f_k\Big)<\frac{1}{n}\]
and thus as previously
\begin{align*}
p_j\Big(\sum_{k=j}^{\infty}a_kf_k-\sum_{k=j}^{\infty}a_{n,k}f_k\Big)
&\le \frac{1}{n} + p_j\Big(\sum_{k=j}^{M}a_kf_k-\sum_{k=j}^{M}a_{n,k}f_k\Big)\\
&=  \frac{1}{n} + \lim_m p_j\Big(\sum_{k=j}^{M}a_{m,k}f_k-\sum_{k=j}^{M}a_{n,k}f_k\Big)\\
&\le \frac{1}{n} + \frac{2K(1+\delta)}{1-\delta} p_j(f-Tx_n)\\
&\quad+ \frac{2K(1+\delta)}{1-\delta} p_j\Big(\sum_{k=1}^{j-1} a_{k} f_k-\sum_{k=1}^{j-1} a_{n,k} f_k\Big)
\xrightarrow{n\rightarrow \infty} 0.
\end{align*}
We conclude that $Tx_n\rightarrow \sum_{k=1}^{\infty}a_kf_k$ as $n\rightarrow \infty$ and thus that $f=\sum_{k=1}^{\infty}a_kf_k$. 

Finally, we notice that $x_n$ converges to $\sum_{k=1}^\infty a_k u_k$ because by \eqref{eq7}, for any $j\ge 1$, 
\begin{align*}
&p_j\Big(\sum_{k=1}^{\infty}a_k u_k -x_n\Big)\\
&\quad\le
p_j\Big(\sum_{k=1}^{j-1}a_k u_k -\sum_{k=1}^{j-1}a_{n,k}u_k\Big)
+p_j\Big(\sum_{k=j}^{\infty}a_k u_k -\sum_{k=j}^{\infty}a_{n,k}u_k\Big)\\
&\quad\le p_j\Big(\sum_{k=1}^{j-1}a_k u_k -\sum_{k=1}^{j-1}a_{n,k}u_k\Big)
+\frac{1}{1-\delta}p_j\Big(\sum_{k=j}^{\infty}a_k f_k -\sum_{k=j}^{\infty}a_{n,k}f_k\Big)\xrightarrow{n\rightarrow \infty} 0.
\end{align*}
We conclude that $T$ is an isomorphism between $M_u$ and $\text{Im}(T)$.
The subspace $M':=\text{Im}(T)$ is thus a closed infinite-dimensional subspace for which $(f_k)$ is a basis. Moreover, we have $M'\cap \ker p_1=\{0\}$ and $(f_k)$ is equivalent to $(u_k)$. That concludes the proof.
\end{proof}

We deduce from Corollary~\ref{corb} and Theorem~\ref{thm bas} that we can construct a basic sequence stable under small perturbations in each infinite-dimensional subspace $M\subset X$ such that for any closed subspace $E$ of finite codimension, \[E\cap M\not\subset \ker p_1.\] If $M$ is closed, we can in fact reformulate this condition as follows:
\begin{prop}
Let $X$ be a Fréchet space, $p$ a continuous seminorm on $X$ and $M$ a closed infinite-dimensional subspace of $X$.
The following assertions are equivalent:
\begin{enumerate}
\item for any closed subspace $E$ of finite codimension in $X$, $E\cap M\not\subset \ker p$,
\item the subspace $M\cap \ker p$ is of infinite codimension in $M$.
\end{enumerate}
\end{prop}
\begin{proof}
We show the equivalence between $\neg (1)$ and $\neg (2)$.

$\neg (1)\Rightarrow \neg (2)$. Let $E\subset X$ be a closed subspace of finite codimension such that
$E\cap M\subset \ker p$. Then the subspace $E\cap M$ is a subspace of finite codimension in $M$ and since
$E\cap M\subset M\cap \ker p\subset M$, we deduce that $M\cap \ker p$ is of finite codimension in $M$.

$\neg (2)\Rightarrow \neg (1)$. We suppose that $M\cap \ker p$ is of finite codimension in $M$. Since $M\cap \ker p$ is closed, there exists $n\ge 1$ and a continuous linear map $f:M\to \K^n$ such that $\ker f=M\cap \ker p$. By Hahn-Banach, there then exists a continuous linear map $\tilde{f}:X\rightarrow \K^n$ such that $\tilde{f}_{|M} =f$. We conclude that $E=\ker \tilde{f}$ is a closed subspace of finite codimension and  that 
\[E\cap M\subset \ker f\subset \ker p.\]
\end{proof}

Thanks to Theorem~\ref{thm bas}, we can extend known criteria about hypercyclic subspaces for Fréchet spaces with a continuous norm to Fréchet spaces without continuous norm by adapting their statement.

The first criterion about hypercyclic subspaces was established by Montes \cite{Montes} in 1996. He has proved that if $T$ is a continuous linear operator on a Banach space $X$ satisfying the Hypercyclicity Criterion for $(n_k)$ and if there exists a closed infinite-dimensional subspace $M_0$ such that for any $x\in M_0$, $(T^{n_k} x)_k$ converges to~$0$, then $T$ possesses a hypercyclic subspace. This criterion has then been generalized to operators on Fréchet spaces with a continuous norm (see \cite{Bonet}, \cite{Petersson}) and to sequences of operators on separable Banach spaces \cite{Leon} and on separable Fréchet spaces with a continuous norm \cite{Menet2} satisfying a certain condition $(C)$.  An elegant alternative proof of this criterion via left-multiplication operators has also been obtained by Chan~\cite{Chan}.

\begin{definition}\label{def C fre}
Let $X$ be a Fréchet space and $Y$ a topological vector space. A sequence $(T_n)$ of operators from $X$ to $Y$ satisfies \emph{condition} (C) if there exist an increasing sequence $(n_k)$ of positive integers and a dense subset $X_0\subset X$ such that
\begin{enumerate}
\item for every $x\in X_0$, $\lim_{k\rightarrow \infty}T_{n_k}x=0$;
\item for every continuous seminorm $p$ on $X$, $\bigcup_k T_{n_k}(\{x\in X: p(x)<1\})$ is dense in $Y$.
\end{enumerate}
\end{definition}

One knows that the above mentioned criterion does not remain true for Fréchet spaces without continuous norm. Indeed, Bonet, Mart\'inez and Peris~\cite{Bonet} have shown that the multiple of the bilateral backward shift $2B$ is mixing on the space $X=\{(x_n)_{n\in \Z}:\sum_{n=0}^{\infty}|x_n|<\infty\}$ endowed with the seminorms
$p_n(x)=\sum_{k=-n}^{\infty}|x_k|$ and does not possess any hypercyclic subspace. However, if we consider $M_0=\{(x_n)_{n\in \Z}:x_n=0 \ \text{for}\ n\ge 1\}$, we remark that for any $x\in M_0$, $((2B)^{k} x)_k$ converges to $0$ as $k\rightarrow \infty$. In fact, the problem in this counterexample does not seem to be the absence of a continuous norm on $X$ but the fact that $M_0\cap \ker p_n$ is a subspace of finite codimension in $M_0$ for any $n\ge 1$. Nevertheless, once $M_0\cap \ker p_1$ is a subspace of infinite codimension in $M_0$, the construction of basic sequences of Theorem \ref{thm bas} allows us to generalize the criterion established by Montes \cite{Montes} to Fréchet spaces without continuous norm.

We fix a Fréchet space $X$, an increasing sequence of seminorms $(p_n)$ defining the topology of $X$ and a separable topological vector space $Y$ whose topology is defined by a sequence of seminorms $(q_n)$. We can then state the following result whose the proof is similar to the proof of \cite[Theorem 1.5]{Menet2}.

\begin{theorem}\label{M0}
Let $(T_n)$ be a sequence of continuous linear operators from $X$ to $Y$. If $(T_n)$ satisfies condition $(C)$ for $(n_k)$ and if there exists a closed infinite-dimensional subspace $M_0$ of $X$ such that 
\begin{enumerate}
\item $M_0\cap \ker p_1$ is a subspace of infinite codimension in $M_0$,
\item for any $x\in M_0$, $T_{n_k} x\rightarrow 0$ as $k\rightarrow \infty$,
\end{enumerate}
then $(T_n)$ possesses a hypercyclic subspace of type $1$.
\end{theorem}
\begin{remark}\label{rem2}
A simple argument (see \cite{Bernal2}, \cite[Remark 10.9]{Karl}) likewise allows us to replace \emph{(2)} by the condition: for any $x\in M_0$, $(T_{n_k} x)_k$ converges in $Y$.
\end{remark}

Le\'on and Müller \cite{Leon} have established another sufficient condition for sequences of operators on Banach spaces to have a hypercyclic subspace. This criterion relies on the existence of a convenient non-increasing sequence of infinite-dimensional closed subspaces $(M_j)$. It has also been generalized to the case of Fréchet spaces with a continuous norm in \cite{Menet2}. By adding a condition on the subspaces $(M_j)$, we can extend this criterion to Fréchet spaces without continuous norm: 

\begin{theorem}\label{Mk}
Let $(T_n)$ be a sequence of continuous linear operators from $X$ to $Y$.
If $(T_n)$ satisfies condition $(C)$ for $(n_k)$ and if there exists a non-increasing sequence of infinite-dimensional closed subspaces $(M_j)_{j\ge 1}$ of $X$
such that 
\begin{enumerate}
\item for any $j\ge 1$, $M_j\cap \ker p_1$ is a subspace of infinite codimension in $M_j$,
\item for each $n\ge 1$, there exist a positive number $C_n$ and two integers $m(n)$, $k(n)\ge 1$ such that for any $j\ge k(n)$, for any $x\in M_j$,
\[q_n(T_{n_j}x)\le C_n p_{m(n)}(x),\]
\end{enumerate}
then $(T_n)$ possesses a hypercyclic subspace of type $1$.
\end{theorem}

On the other hand, a criterion for having no hypercylic subspace is given in \cite{Leon} for sequences of operators on Banach spaces. It has been adapted for operators on Fréchet spaces in \cite{Karl} and for sequences of operators on Fréchet spaces in \cite{Menet2}. These results directly give us a condition for having no hypercyclic subspaces of type $1$.

\begin{theorem}
Let $(T_n)$ be a sequence of continuous linear operators from $X$ to $Y$.
If there exist a sequence $(E_n)$ of closed subspaces of finite codimension in $X$, positive numbers $C_n\rightarrow \infty$ and a continuous seminorm $q$ on $Y$ such that for any $n\ge 1$, for any $x\in E_n$, we have
\[q(T_{n}x)\ge C_np_n(x),\]
then $(T_n)$ does not possess any hypercyclic subspace of type $1$.
\end{theorem}
\begin{proof}
Suppose that $M$ is a hypercyclic subspace of type $1$.
By definition, there then exists $J\ge 1$ such that for any non-zero vector $x\in M$, we have $p_{J}(x)\ne 0$. Therefore, the proof of \cite[Theorem 1.13]{Menet2} allows us to construct a vector $x\in M$ such that ${q(T^nx)\rightarrow \infty}$ as $n\rightarrow \infty$: this is a contradiction with the fact that $M$ is a hypercyclic subspace.
\end{proof}

Thanks to a simplification of the previous criterion in the case of an operator $T:X\rightarrow X$ given by the author in \cite{Menet2}, we obtain the following simpler criterion:

\begin{theorem}\label{thm E}
Let $T$ be a continuous linear operator from $X$ to itself. If there exists a continuous seminorm $q$ on $X$ such that for any $n\ge 1$, there exist a closed subspace $E$ of finite codimension in $X$, $C>1$ and $m\ge 1$ such that for any $x\in E$,
\[q(T^{m}x)\ge Cp_n(x),\]
then $T$ does not possess any hypercyclic subspace of type $1$.
\end{theorem}

\section{Existence of hypercyclic subspaces}
\subsection{Differential operators on $C^{\infty}(\R)$}\label{cinfty}
Let $C^{\infty}(\R)$ be the Fréchet space of infinitely differentiable functions on $\R$ endowed with the increasing sequence of seminorms
\[p_k(f):=\sum_{l=0}^ {k-1} \sup_{x\in[-k,k]}|f^{(l)}(x)|,\]
where we denote by $f^{(l)}$ the $l$th derivative of $f$.
The space $C^{\infty}(\R)$ is an important example of Fréchet spaces without continuous norm and one can wonder if the differential operators possess hypercyclic subspaces on this space.

In the case of the space of entire functions $H(\C)$, we know that each differential operator $\varphi(D)$, where $D$ is the derivative operator and $\varphi$ is a non-constant entire function of exponential type, possesses a hypercyclic subspace in $H(\C)$. This result has been proved by Petersson~\cite{Petersson} for entire functions which are not polynomials, by Shkarin~\cite{Shkarin} for the operator $D$ and completed by the author~\cite{Menet2} in the case of non-constant polynomials. In the case of the space $C^{\infty}(\R)$, we remark that the map $\varphi(D)$ will be a well-defined operator on $C^{\infty}(\R)$ if and only if $\varphi$ is a polynomial. Indeed, by Borel theorem~\cite{Borel}, we know that for any sequence $(a_n)_{n\ge 0}\in \R^{\Z_+}$, there exists a function $f\in C^{\infty}(\R)$ such that for any $n\ge 0$, $f^{(n)}(0)=a_n$.

Let $P$ be a non-constant polynomial. We already know thanks to Godefroy and Shapiro~\cite{Godefroy} that the operator $P(D)$ is hypercyclic on $C^{\infty}(\R)$. We now show that for every non-constant polynomial $P$, the operator $P(D)$ possesses a hypercyclic subspace of type~$1$ in $C^{\infty}(\R)$.

\begin{lemma}\label{lemmasup}
Let $T$ be an operator on $C^{\infty}(\R)$. If for any $k\ge 1$, \[T(\ker p_k)\subset \ker p_k,\] then for any closed subspace $E$ of finite codimension in $C^{\infty}(\R)$, any $\varepsilon>0$, any $N\ge 1$, there exists $f\in E$ such that $p_1(f)=1$ and for any $1\le l,k\le N$,
\[p_k(T^l f)\le \varepsilon p_{k+1}(f).\]
\end{lemma}
\begin{proof}
Let $E$ be a closed subspace of finite codimension in $C^{\infty}(\R)$ and $d$ the codimension of $E$. Let $\varepsilon>0$ and $N\ge 1$. We let $I_{j,k}=]k+\frac{j}{d+1},k+\frac{j+1}{d+1}[$ for any $0\le j\le d$, any $0\le k\le N$. For any $0\le j\le d$, we then choose $f_{j,0},\dots, f_{j,N}\in C^{\infty}(\R)$ such that
\begin{itemize}
\item for any $0\le k\le N$, $\text{supp}(f_{j,k})\subset I_{j,k}$;
\item $p_1(f_{j,0})=1$;
\item for any $1\le l,k\le N$,
\[
p_k\Big(T^l\Big(\sum_{i=0}^{k-1}f_{j,i}\Big)\Big)\le \varepsilon \sup_{x\in I_{j,k}} |f_{j,k}(x)|.
\]
\end{itemize}
Let $f_j=\sum_{k=0}^N f_{j,k}, j=0,\dots, d$. Since for any $k\ge 1$, $T(\ker p_k)\subset \ker p_k$ and $f_{j,k}\in \ker p_k$, we deduce that for any $0\le j\le d$, any $1\le l,k\le N$,
\begin{align}
p_k(T^lf_j)= p_k\Big(T^l\Big(\sum_{i=0}^{k-1}f_{j,i}\Big)\Big) &\le \varepsilon \sup_{x\in I_{j,k}} |f_{j,k}(x)|\nonumber\\
&\le \varepsilon \sup_{x\in [-k-1,k+1]} |f_{j}(x)|\le \varepsilon p_{k+1}(f_j).
\label{max sup}
\end{align}
Since $E$ is a subspace of finite codimension $d$ and $f_0,\dots, f_d$ are clearly linearly independent, there exists $a_0,\dots, a_d$ with $\max\{|a_j|:0\le j\le d\}=1$ such that $f:=a_0f_0+\cdots+a_df_d\in E$. Since the supports of the $f_j$ are disjoint, we have 
\[p_1(f)=\max_{0\le j\le d}|a_j|p_1(f_j)=\max_{0\le j\le d}|a_j|p_1(f_{j,0})=1.\]
Moreover, we deduce from \eqref{max sup} that for any $1\le l,k\le N$,
\begin{align*}
p_k(T^l f)&\le \sum_{j=0}^{d}|a_j| p_k(T^l f_j)\le \varepsilon \sum_{j=0}^{d}|a_j| p_{k+1}(f_j)\\
&\le (d+1)\varepsilon \max_{0\le j\le d} |a_j| p_{k+1}(f_j)\le (d+1)\varepsilon p_{k+1}(f).
\end{align*}
\end{proof}
\begin{theorem}\label{thmCinf}
Let $T$ be an operator on $C^{\infty}(\R)$. If $T$ satisfies condition $(C)$ and if for any $k\ge 1$, $T(\ker p_{k})\subset \ker p_k$, then $T$ possesses a hypercyclic subspace of type~$1$.
\end{theorem}
\begin{proof}
Let $(n_k)_{k\ge 1}$ be an increasing sequence of positive integers such that $T$ satisfies condition $(C)$ for $(n_k)$ and $(\varepsilon_n)_{n\ge 1}$ a sequence of positive real numbers such that $\prod_{n} (1+\varepsilon_n)\le 2$.
Thanks to Lemma~\ref{lem bas} and Lemma~\ref{lemmasup}, we know that there exists a sequence $(u_k)_{k\ge 1}\subset C^{\infty}(\R)$ satisfying \eqref{lem 1} for the sequence $(\varepsilon_n)_{n\ge 1}$ and such that for any $1\le j,n\le k$,
\begin{equation}
p_n(T^{n_j} u_k)\le \frac{1}{2^k}p_{n+1}(u_k).
\label{eqcin}
\end{equation}
We show that $T$ satisfies the assumptions of Theorem~\ref{Mk} for the sequence of closed infinite-dimensional subspaces
$M_j=\overline{\text{span}}\{u_k:k\ge j\}$, $m(n)=n+1$, $k(n)=n+2$ and $C_n=1$.

By Theorem~\ref{thm bas}, we know that
$\overline{\text{span}}\{u_k:k\ge 1\}\cap \ker p_1=\{0\}$. The subspace $M_j\cap \ker p_1$ is thus of infinite codimension in $M_j$. On the other hand, since $(u_n)_{n\ge 1}$ is a basic sequence (Theorem~\ref{thm bas}), we know that for any $x\in M_j$, we have $x=\sum_{k=j}^{\infty}a_k u_k$ for some sequence $(a_k)_{k\ge 1}\in \K^{\N}$ and we deduce from \eqref{lem 1} and \eqref{eqcin} that
for any $n\ge 1$, any $j\ge n+2$, any $x\in M_j$, $x=\sum_{k=j}^{\infty}a_k u_k$, we have 
\begin{align*}
p_n(T^{n_j}x)&\le\sum_{k=j}^{\infty}p_n(T^{n_j}(a_ku_k))
\le \sum_{k=j}^{\infty}\frac{1}{2^k}p_{n+1}(a_ku_k)\\
&\le \sum_{k=j}^{\infty}\frac{4}{2^k}p_{n+1}(x)
\le p_{n+1}(x).
\end{align*}
\end{proof}
\begin{cor}
For any non-constant polynomial $P$, the operator $P(D)$ possesses a hypercyclic subspace of type~$1$ in $C^{\infty}(\R)$.
\end{cor}
\begin{proof}
Let $P$ be a non-constant polynomial.
We know that the operator $P(D)$ satisfies the Hypercyclicity Criterion~\cite{Godefroy}. Therefore, since for any $k\ge 1$, \[P(D)(\ker p_{k})\subset \ker p_k,\]
we deduce from Theorem~\ref{thmCinf} that $P(D)$ possesses a hypercyclic subspace of type~$1$ in $C^{\infty}(\R)$.
\end{proof}
\begin{remark}
For any non-constant polynomial $P$, the operator $P(D)$ does not possess any hypercyclic subspace of type~$2$ in $C^{\infty}(\R)$ since there is no hypercyclic vector for $P(D)$ in $\ker p_1$.
\end{remark}

\subsection{Universal series}\label{section univ}

Let $A$ be a Fréchet space of sequences such that the coordinate projections $P_m:A\rightarrow\K,$ $(a_n)_{n\ge 0}\mapsto a_m$
are continuous for all $m\ge 0$, and the set of polynomials $\{a=(a_n)_{n\geq 0}\in\K^{\Z_+}\ :\{n\ge 0:\ a_n\ne 0\} \hbox{ is finite}\}$ is contained  and dense in $A.$ 
We denote by $(e_n)_{n\geq 0}$ the canonical basis of $\K^{\Z_+}$ and by $v(a)$ and $d(a)$ the valuation and the degree of the polynomial $a\in A$.

Let $X$ be a separable topological vector space whose topology is given by an increasing sequence of seminorms $(q_j)$ and let $(x_n)_{n\ge 0}$ be a sequence in $X$. We let $S_k:A\rightarrow X$ be the operator defined by \[S_k((a_n)_{n\ge 0})=\sum_{n=0}^{k}a_nx_n\] and $\mathcal{U}\cap A$ the set of hypercyclic vectors for $(S_k)$. A sequence $a\in\mathcal{U}\cap A$ is called a \emph{universal series}. An interesting subclass of universal series is the class of restricted universal series.

\begin{definition}
Let $S^A_k:A\rightarrow A$ be the operator defined by \[S^A_k((a_n)_{n\ge 0})=\sum_{n=0}^{k}a_ne_n.\]
A sequence $a\in A$ is a \emph{restricted 
universal series} if, for every $x\in X$,
there exists an increasing sequence $(n_k)$ in $\N$ such that
\begin{equation} S_{n_k}a\rightarrow x \text{ in $X$ and }
S^A_{n_k}a\rightarrow a\text{ in $A$ as }k\rightarrow
\infty. \end{equation}
We denote by $\mathcal{U}_A$ the set of such series.
\end{definition}
For more details about universal series, we refer to the article of Bayart, Grosse-Erdmann, Nestoridis and Papadimitropoulos~\cite{Bayart2}.\\

In 2010, Charpentier \cite{Charpentier} proved that if $\mathcal{U}_A\ne \emptyset$ and $A$ is a Banach space then $\mathcal{U}_A$ is spaceable i.e. $\mathcal{U}_A\cup\{0\}$ contains a closed infinite-dimensional subspace. Using the construction of basic sequences in Fréchet spaces with a continuous norm, the author \cite{Menet} has generalized this result to Fréchet spaces with a continuous norm. The proof of these results depends on three factors: the existence of a restricted universal series, the possibility to construct a basic sequence of polynomials with valuation as large as desired and the fact that a sufficiently small perturbation of this basic sequence remains basic. We prove that, thanks to Theorem \ref{thm bas}, these conditions are satisfied in the case where $\mathcal{U}_A\ne \emptyset$ and $A$ is a Fréchet space admitting a continuous seminorm $p$ such that $\ker p$ is not a subspace of finite codimension.

\begin{theorem}\label{univ}
If $\mathcal{U}_A$ is non-empty and $A$ is a Fréchet space with a continuous seminorm $p$ such that $\ker p$ is not a subspace of finite codimension, then $\mathcal{U}_A$ is spaceable.
\end{theorem}
\begin{proof}
The proof is similar to the proof for Fréchet spaces with a continuous norm \cite{Menet}. Therefore, we only justify the possibility to construct a basic sequence of polynomials with valuation as large as desired thanks to Corollary~\ref{corb} and Theorem~\ref{thm bas}. To this end, we prove that the subspace $M_k$ of polynomials with valuation at least $k$ satisfies
\[E\cap M_k\not\subset \ker p \quad\text{for any closed subspace $E$ of finite codimension}.\] Since $<e_0,\dots,e_{k-1}>\oplus M_k$ is dense, we have \[A=\overline{<e_0,\dots,e_{k-1}>\oplus M_k}= <e_0,\dots,e_{k-1}>\oplus \overline{M_k}.\] The subspace $\overline{M_k}$ is thus of finite codimension in $A$.
Let $E$ be a closed subspace of finite codimension in $X$. There exists a finite-dimensional subspace $F$ such that $(E\cap M_k)\oplus F=M_k$. We deduce that
\[\overline{M_k}=\overline{(E\cap M_k)\oplus F}=\overline{(E\cap M_k)}\oplus F.\]
Since $\overline{M_k}$ is a subspace of finite codimension in $A$, we conclude that $\overline{E\cap M_k}$ is a subspace of finite codimension in $A$ and thus that we cannot have
$E\cap M_k\subset \ker p$ as $\ker p$ is not a subspace of finite codimension.
\end{proof}

\subsection{Existence of mixing operators with hypercyclic subspaces}\label{existence}

The existence of hypercyclic and even mixing operators on any infinite-dimensional separable Fréchet space is now well known. In fact, the existence of hypercyclic operators has been obtained on any infinite-dimensional separable Banach space by Ansari \cite{Ansari} and Bernal \cite{Bernal}, and on any infinite-dimensional separable Fréchet space by Bonet and Peris \cite{Bonet2}. The construction of these operators is based on the hypercyclity of perturbations of the identity by a weighted shift. Grivaux has then noticed in \cite{Grivaux} that these operators are even mixing.

In 1997, Le\'on and Montes \cite{Leon0} proved that every infinite-dimensional separable Banach space supports an operator with a hypercyclic subspace. In 2006, this result was generalized for infinite-dimensional separable Fréchet spaces with a continuous norm independently by Bernal \cite{Bernal2} and Petersson \cite{Petersson}. For Fréchet spaces isomorphic to $\omega$, the same result was obtained by B\`{e}s and Conejero \cite{Bes}. Theorem \ref{M0} permits us to answer positively the question asked in \cite[Problem 8]{Bes}: "Does every separable infinite-dimensional Fréchet space support an operator with a hypercyclic subspace?".

\begin{lemma}[\cite{Bonet2}]\label{lemexi}
Let $X$ be a separable infinite-dimensional Fréchet space.
If $X$ is not isomorphic to $\omega$ then 
there exists a dense subspace $F$ of $X$ which has a continuous norm $p$ and there are sequences $(x_n)_n\subset F$ and $(f_n)_n\subset X'$, the dual of $X$, such that
\begin{enumerate}
\item $(x_n)_n$ converges to $0$, and $\text{\emph{span}}\{x_n:n\in \N\}$ is dense in $X$;
\item $(f_n)_n$ is $X$-equicontinuous in $X'$;
\item $\langle x_n,f_m\rangle=0$ if $n\ne m$ and $(\langle x_n,f_n\rangle)_n\subset\ ]0,1[$.
\end{enumerate}
\end{lemma}

\begin{theorem}
Every separable infinite-dimensional Fréchet space supports a mixing operator with a hypercyclic subspace.
\end{theorem}
\begin{proof}
Let $B:\omega\rightarrow \omega$ be the backward shift. Since the operator $B$ is a mixing operator with a hypercyclic subspace in $\omega$ (see \cite{Bes}), any space $X$ isomorphic to $\omega$ supports a mixing operator with a hypercyclic subspace.

Now, if $X$ is not isomorphic to $\omega$, then  Lemma \ref{lemexi} permits us to construct a mixing operator $T:X\rightarrow X$ (see \cite{Bonet2,Grivaux}) defined by
\[Tx=x+\sum_{n=1}^{\infty}2^{-n}\langle x,f_{2n}\rangle x_n.\]
In particular, since $T$ is mixing, $T$ satisfies the Hypercyclicity Criterion and thus condition $(C)$. 

Let $F$ and $p$ as given by Lemma \ref{lemexi}. There exists a continuous seminorm $q$ on $X$ such that the restriction of $q$ to $F$ coincides with $p$. The subspace $M_0:=\overline{\text{span}}\{x_n:n \ \text{odd}\}$ then satisfies the assumptions of Theorem~\ref{M0} with the modification introduced in Remark~\ref{rem2}, since each vector in $M_0$ is a fixed point and 
\[\text{span}\{x_n:n \ \text{odd}\}\cap \ker q=\{0\}.\]
We conclude by Theorem~\ref{M0}.
\end{proof}

\section{Some criteria for hypercyclic subspaces of type 2}\label{type2}

We move from our study of hypercyclic subspaces of type $1$ to the study of hypercyclic subspaces of type $2$.
In order to construct hypercyclic subspaces of type $2$, we first have to find a way to construct basic sequences $(u_n)$ in $X$ such that if we denote by $M_u$ the closed linear span of $(u_n)$, then $M_u\cap \ker p$ is infinite-dimensional for any continuous seminorm $p$ on $X$.

\begin{lemma}\label{bas ker}
Let $X$ be a Fréchet space, $(p_n)$ an increasing sequence of seminorms inducing the topology of $X$ and $(u_n)_{n\ge 1}\subset X$. If for any $n\ge 1$, ${u_n\in \ker p_n\backslash \ker p_{n+1}}$, then $(u_n)$ is a basic sequence,
\[M_u:=\overline{\text{\emph{span}}}\{u_n:n\ge 1\}=\Big\{\sum_{n=1}^{\infty}\alpha_nu_n:(\alpha_n)\in \K^{\N}\Big\}\]
and for any $k\ge 1$,
\[M_u\cap \ker p_k=
\overline{\text{\emph{span}}}\{u_n:n\ge k\}=
\Big\{\sum_{n=k}^{\infty}\alpha_nu_n:(\alpha_n)\in \K^{\N}\Big\}.\]
In particular, the space $M_u\cap \ker p$ is infinite-dimensional for any continuous seminorm $p$ on $X$.
\end{lemma}
\begin{proof}
Let $(u_n)$ be a sequence in $X$ such that for any $n\ge 1$, $u_n\in \ker p_n\backslash \ker p_{n+1}$. For any sequence $(\alpha_n)\in \K^{\N}$, the series $\sum_{n=1}^{\infty} \alpha_n u_n$ converges because for any $k\ge 1$, any $M\ge N\ge k$, we have $p_k\big(\sum_{n=N}^{M} \alpha_n u_n\big)=0$. Moreover, if $x=\sum_{n=1}^{\infty} \alpha_n u_n$ and $x=\sum_{n=1}^{\infty} \beta_n u_n$, then for any $n\ge 1$, $\alpha_n=\beta_n$. Indeed, suppose that $n_0$ is the smallest index such that $\alpha_{n_0}\ne\beta_{n_0}$, then we get the following contradiction:
\[0=p_{n_0+1}\Big(\sum_{n=1}^{\infty} \alpha_n u_n-\sum_{n=1}^{\infty} \beta_n u_n\Big)=|\alpha_{n_0}-\beta_{n_0}|p_{n_0+1}(u_{n_0})\ne 0.\]
Let $x\in X$ and $x_k=\sum_{n=1}^{\infty} \alpha_{n,k} u_n$. To finish the proof, it suffices to show that if the sequence $(x_k)_k$ converges to $x$ in $X$, then each sequence $(\alpha_{n,k})_k$ converges to some scalar $\alpha_n$ and $x=\sum_{n=1}^{\infty} \alpha_n u_n$.
We first notice that since $p_2(x_k-x_j)=|\alpha_{1,k}-\alpha_{1,j}|p_2(u_1)$ converges to $0$ as $k,j\rightarrow \infty$, the sequence $(\alpha_{1,k})_k$ is a Cauchy sequence and thus converges to some scalar $\alpha_1$.
Then we remark by induction that for any $N\ge 2$, 
the sequence $(\alpha_{N,k})_k$ is a Cauchy sequence and thus converges to some scalar $\alpha_{N}$. Therefore, we have, for any $N\ge 2$,
\begin{align*}
p_N\Big(x-\sum_{n=1}^{\infty} \alpha_n u_n\Big)&\le p_N(x-x_k)+p_N\Big(x_k-\sum_{n=1}^{\infty} \alpha_n u_n\Big)\\
&= p_N(x-x_k)+p_N\Big(\sum_{n=1}^{N-1} \alpha_{n,k} u_n-\sum_{n=1}^{N-1} \alpha_n u_n\Big)\xrightarrow{k\rightarrow \infty} 0.
\end{align*}
\end{proof}
\begin{remark}
Such a basic sequence does not remain basic under small perturbations; for example in $\omega$, $(e_k)_{k \ge 0}$ is a basic sequence but $(e_k+\varepsilon_ke_{k-1})_{k\ge 0}$, with $e_{-1}=0$, is not basic for any $\varepsilon_k>0$.
\end{remark}

Let $(T_k)$ be a sequence of linear operators from $X$ to $Y$, where $X$ is an infinite-dimensional Fréchet space without continuous norm, whose topology is given by an increasing sequence of seminorms $(p_n)$, and where $Y$ is a separable topological vector space, whose topology is given by an increasing sequence of seminorms $(q_j)$.
Using the previous construction of basic sequences, we can state a sufficient criterion for the existence of hypercyclic subspaces of type $2$.

\begin{theorem}\label{thm X0}
If there exist an increasing sequence $(n_k)$ and a set $X_0\subset X$ such that 
\begin{enumerate}[\upshape 1.]
\item for any $j\ge 1$, for any $x\in X_0$, $(q_j(T_{n_k}x))_k$ is ultimately zero;
\item for any $n,K\ge 1$, $\bigcup_{k\ge K} T_{n_k}(X_0\cap\ker p_n)$ is dense in $Y$,
\end{enumerate}
then $(T_n)$ possesses a hypercyclic subspace of type $2$.
\end{theorem}
\begin{proof}
If for any continuous seminorm $q$ on $Y$, for any $y\in Y$, $q(y)=0$, then 
$X$ is a hypercyclic subspace of type $2$ because, by our assumption, $X$ does not admit a continuous norm.
If this is not the case, there exists a dense sequence $(y_l)_{l\ge 1}$ in $Y$ such that for any $l\ge 1$, $q_j(y_l)\ne 0$ for some $j\ge 1$. Without loss of generality, we can suppose that $q_l(y_l)\ne 0$. By continuity of $T_{n_k}$, we also notice that for any $K\ge 1$, there exists $N_K\ge 1$ and $C>0$ such that for any $x\in X$,
\[\max_{k\le K}q_{k}(T_{n_k}x)\le C p_{N_K}(x).\]

In order to use Lemma \ref{bas ker}, we then seek to construct a sequence $(u_k)\subset X$ such that for any $k\ge 1$, $u_k\in \ker p_k\backslash \ker p_{k+1}$ and for any $(\alpha_k)_k\subset \K^{\N}\backslash\{0\}$, the series $\sum_{k=1}^{\infty}\alpha_k u_k$ is hypercyclic for $(T_{n_k})_k$. By hypothesis, for any $\varepsilon>0$, for any $l,K,n\ge 1$, there exist $x\in X_0\cap \ker p_n$ and $k\ge K$ such that 
\[q_l(T_{n_k}x-y_l)<\varepsilon\]
and since $q_l(y_l)\ne 0$, we can assume that $x\ne 0$ so that there exists $m>n$ such that $p_m(x)\ne 0$.
For any $l,n,K_0\ge 1$, there then exist $x\in X_0\cap\ker p_n$ and an increasing sequence $(K_j)_{j\ge 1}$ such that
\begin{itemize}
\item for any $i\le K_0$,\quad $q_{i}(T_{n_i}x)=0$ (choosing $x\in \ker p_{N_{K_0}}$);
\item there exists $K_0\le i<K_{1}$, such that $q_l(T_{n_i}x-y_l)< \frac{1}{l}$;
\item for any $j\ge 1$, any $i\ge K_{j}$, $q_j(T_{n_i}x)=0$ (because $x\in X_0$).
\end{itemize}

We construct recursively a family $(x_{k,l})_{k,l\ge 1}\subset X_0$ and a family $(n_{\phi(k,l)})_{k,l\ge 1}$, for the strict order $\prec$ defined by $(k',l')\prec(k,l)$ if $k'+l'<k+l$ or if $k'+l'=k+l$ and $l'<l$, such that
\begin{itemize}
\item for any $k\ge 1$, $(\phi(k,l))_l$ is strictly increasing;
\item for any $k,l\ge 1$, $x_{k,l}\in \ker p_{k+l}$ and
$q_{l}(T_{n_{\phi(k,l)}}x_{k,l}-y_l)<\frac{1}{l}$;
\item for any $k,l\ge 1$, there exists an increasing sequence $(K_{j}^{k,l})_{j\ge 0}$ such that
\begin{enumerate}[\upshape (a)]
\item $K_0^{k,l}\ge l$;
\item $\phi(k,l)\in [K_0^{k,l},K_{1}^{k,l}[$;
\item for each pair $(k',l')\prec (k,l)$, $K_{l}^{k',l'}<K_0^{k,l}$;
\item for any $i\le K_0^{k,l}$, $q_{i}(T_{n_i}x_{k,l})=0$;
\item for any $j\ge 1$, any $i\ge K_{j}^{k,l}$, $q_j(T_{n_i}x_{k,l})=0$.
\end{enumerate}
\end{itemize}
Since we use the seminorm $p_{k+2}$ only after the construction of $x_{k,1}$, we may assume that $p_{k+2}(x_{k,1})\ne 0$, by changing the seminorm $p_{k+2}$ if necessary.

Let $u_k=\sum_{l=1}^{\infty} x_{k,l}$ which clearly converges. We therefore remark that we have 
$p_{k+1}(u_k)=0$ and $p_{k+2}(u_k)\ne 0$, and thus by Lemma \ref{bas ker} \[M:=\overline{\text{span}}\{u_k:k\ge 1\}=\Big\{\sum_{k=1}^{\infty}\alpha_ku_k:(\alpha_k)\in\mathbb{K}^{\mathbb{N}}\Big\}.\] 
Moreover, by our construction, we notice that for any $l_0\ge1$:
\begin{itemize}
\item if $(k,l)\prec (k_0,l_0)$ then, by (b) and (c), we have $K_{l_0}^{k,l}<K_0^{k_0,l_0}\le \phi(k_0,l_0)$ and thus, by (e), $q_{l_0}(T_{n_{\phi(k_0,l_0)}}x_{k,l})=0$;
\item if $(k_0,l_0)\prec (k,l)$ then, by (a), (b) and (c), we have \[l_0\le K_0^{k_0,l_0}\le \phi(k_0,l_0)< K_l^{k_0,l_0}< K_0^{k,l}\] and thus, by (d), $q_{l_0}(T_{n_{\phi(k_0,l_0)}}x_{k,l})\le q_{\phi(k_0,l_0)}(T_{n_{\phi(k_0,l_0)}}x_{k,l})=0$.
\end{itemize}
Therefore, if $u=\sum_{k=1}^{\infty}\alpha_k u_k$ with $\alpha_{k_0}=1$, we have for any $l_0\ge 1$:
\[q_{l_0}(T_{n_{\phi(k_0,l_0)}}u-y_{l_0})=q_{l_0}(T_{n_{\phi(k_0,l_0)}}x_{k_0,l_0}-y_{l_0})<\frac{1}{l_0}.\]
We conclude that each non-zero vector in $M$ is hypercyclic and thus that $M$ is a hypercyclic subspace of type $2$.
\end{proof}
\begin{example}
Every non-constant translation operator on $C^{\infty}(\R)$ possesses a hypercyclic subspace of type $2$. This result directly follows from Theorem~\ref{thm X0} by considering the set of infinitely differentiable functions with finite support.
\end{example}

We can simplify the previous criterion if each continuous seminorm on $X$ has a kernel of finite codimension.

\begin{cor}\label{cor X0}
Suppose that each seminorm $p_n$ has a kernel of finite codimension.
If there exist an increasing sequence $(n_k)$ and a set $X_0\subset X$ such that 
\begin{enumerate}[\upshape 1.]
\item for any $j\ge 1$, any $x\in X_0$, $(q_j(T_{n_k}x))_k$ is ultimately zero;
\item for any $K\ge 1$, $\bigcup_{k\ge K} T_{n_k}(X_0)$ is dense in $Y$.
\end{enumerate}
then $(T_n)$ possesses a hypercyclic subspace.
\end{cor}
\begin{proof}
Thanks to Theorem \ref{thm X0}, it is sufficient to prove that for any $n,K\ge 1$, for any non-empty open set $U\subset Y$, there exists $x\in\ker p_n$ such that
\begin{enumerate}
\item for any $j\ge 1$, $(q_j(T_{n_k}x))_k$ is ultimately zero;
\item for some $k\ge K$, $T_{n_k}x\in U$.
\end{enumerate}
Let $n,K\ge 1$ and $U$ a non-empty open set in $Y$. We denote by $d$ the codimension of $\ker p_n$. If for any continuous seminorm $q$ on $Y$, for any $y\in Y$, $q(y)=0$, then $X$ is a hypercyclic subspace of type $2$. If this is not the case, each non-empty open set in $Y$ contains a vector $y$ satisfying $q(y)\ne 0$ for some continuous seminorm $q$ on $Y$ and there thus exist $y_0\in Y$, $j_0\ge 1$ and $\varepsilon>0$ such that 
\[\{y\in Y:q_{j_0}(y_0-y)<\varepsilon\} \subset U \quad\text{and}\quad \varepsilon<q_{j_0}(y_0).\]
By hypothesis, there exists $x_1\in X_0$ such that for some $k_1\ge K$, \[T_{n_{k_1}}x_1\in \{y\in Y:q_{j_0}(y_0-y)<\varepsilon\}.\] 
Since $x_1\in X_0$, there also exists $K_1\ge k_1$ such that for any $k\ge K_1$, $q_{j_0}(T_{n_k}x_1)=0$. In particular, for any $k\ge K_1$, $T_{n_k}x_1\notin \{y\in Y:q_{j_0}(y_0-y)<\varepsilon\}$. Therefore, since $\bigcup_{k\ge K_1} T_{n_k}(X_0)$ is dense in $X$, there exists another vector $x_2\in X_0$ such that for some $k_2\ge K_1$, \[T_{n_{k_2}}x_2\in \{y\in Y:q_{j_0}(y_0-y)<\varepsilon\}.\]
We can thus find $x_1,\dots,x_{d+1}\in X_0$ and $k_1\le\cdots\le k_{d+1}$ such that for any $1\le i\le d+1$, $T_{n_{k_i}} x_i\in U$ and for any $1\le i< j\le d+1$, $q_{j_0}(T_{n_{k_j}} x_i)=0$. In particular, $x_1,\dots, x_{d+1}$ are linearly independent because for any $1\le j\le d+1$, any $a_1,\dots,a_{j}\in \K$ with $a_j\ne0$,
\[q_{j_0}(T_{n_{k_{j}}}(a_1x_1+\cdots+a_jx_j))=|a_j|q_{j_0}(T_{n_{k_{j}}}x_j)\ne 0.\]
Since $\ker p_n$ has codimension $d$, we deduce that there exist $a_1,\dots,a_{d+1}$ such that $x:=a_1x_1+\cdots+a_{d+1}x_{d+1}\in \ker p_n\backslash\{0\}$.
Therefore, since $x_1,\dots,x_{d+1}\in X_0$, $(q_j(T_{n_k}x))_k$ is ultimately zero  for any $j\ge 1$ and if we let $i_0=\max\{1\le i\le d+1:a_i\ne 0\}$, without loss of generality $a_{i_0}=1$ and $T_{n_{k_{i_0}}}x\in U$. The result follows.
\end{proof}

\begin{samepage}
Now we establish a sufficient criterion for not having a hypercyclic subspace of type~$2$.

\begin{theorem}\label{thm E2}
Let $X_{j}=\{x\in X: \#\{k:q_j(T_{k}x)=0\}=\infty\}$.
If there exist $j,N,K\ge 1$ and a subspace $E$ of finite codimension in $X$ such that 
\[\bigcup_{k\ge K}T_k\big(\ker p_N\cap E \cap X_j\big) \text{\quad is not dense in $Y$,}\]
then $(T_k)$ does not possess any hypercyclic subspace of type $2$.
\end{theorem}
\end{samepage}
\begin{proof}
By hypothesis, there exist a non-empty open set $U\subset Y$, $j,N,K\ge 1$ and a subspace of finite codimension $E$ in $X$ such that
\[\Big(\bigcup_{k\ge K}T_k\big(\ker p_N\cap E \cap X_j\big)\Big)\cap U=\emptyset.\]
A direct consequence is that for any $x\in \ker p_N\cap E$, if there exists $k\ge K$ such that $T_kx\in U$, then $(q_j(T_k x))$ is ultimately non-zero.
Suppose that there exists a hypercyclic subspace $M$ of type $2$. Then we can consider a sequence $(u_n)_{n\ge 1}$ of non-zero vectors such that for any $n\ge 1$, $u_n\in M\cap\ker p_{n+N}\cap E$.
In particular, for any $n\ge 1$, $u_n$ is hypercyclic and there thus exists $k_n\ge K$ such that $T_{k_n}u_n\in U$. Since $u_n\in \ker p_N\cap E$, we deduce by our previous reasoning that there exists $K_n\ge 1$ such that for any $k>K_n$, \[q_j(T_k u_n)\ne 0.\]
We seek to construct a sequence $(\alpha_n)_{n\ge 1}\in \mathbb{R}^{\mathbb{N}}$ such that if we let $u=\sum_{n=1}^{\infty} \alpha_n u_n$, which exists because $u_n\in \ker p_{n+N}$, then, for any $k> K_1$, we have $q_j(T_k(u))\ge 1$. That will give us the desired contradiction as $u\in M\backslash\{0\}$ and $u$ is not hypercyclic. To this end, we start by choosing $\alpha_1$ such that for any $K_1<k\le K_2$, we have
\[q_j(\alpha_1T_k u_1)>1.\]
Then we choose $\alpha_n$ recursively such that for any $K_1<k\le K_{n+1}$, we have
\[q_j\Big(T_k\Big(\sum_{l=1}^{n-1}\alpha_l u_l\Big)+\alpha_n T_k u_n\Big)>1.\]
Such a choice exists because for any $K_n< k\le K_{n+1}$, we have $q_j(T_k(u_n))\ne 0$ and for any $K_1<k\le K_n$, either we have $q_j(T_k(u_n))\ne 0$ and we have just to choose $\alpha_n$ sufficiently large, or we have $q_j(T_k(u_n))= 0$ and we have the desired inequality by the induction hypothesis. We conclude by continuity of $(T_k)$ and $q_j$.
\end{proof}

\section{Some examples of hypercyclic subspaces of type 2}\label{extype2}

\subsection{Universal series}\label{univ series}

We refer to Section \ref{section univ} for notations and definitions about universal series.
We only recall that $A$ is a Fréchet space of sequences whose topology is given by an increasing sequence of seminorms $(p_n)_{n\ge 1}$, $X$ is a separable topological vector space whose topology is given by an increasing sequence of seminorms $(q_j)_{j\ge 1}$, $(x_n)_{n\ge 0}$ is a sequence in $X$ and $S_k:A\rightarrow X$ is the operator defined by \[S_k((a_n)_{n\ge 0})=\sum_{n=0}^{k}a_nx_n.\] 

\begin{samepage} Thanks to the criteria of Section \ref{type2}, we obtain the following two results:

\begin{theorem}\label{condpas}
Suppose that for any $n\ge 1$, $\ker p_n$ is a subspace of finite codimension.
If for any $N\ge 0$,
\[\bigcup_{M\ge N}\left( \text{\upshape{span}}\{x_k:N\le k\le M \}\cap \text{\upshape{span}}\{x_k:M+1\le k\}\right)\] is dense in $X$, then the sequence $(S_k)$ possesses a hypercyclic subspace.
\end{theorem}
\end{samepage}
\begin{proof}
Let $y\in X$, $N\ge 0$ and $l\ge 1$. There exist $N\le M<L$ and $a_N,\dots,a_L\in \K$
such that
\[q_l\Big(\sum_{k=N}^Ma_kx_k-y\Big)<\frac{1}{l} \quad \text{and}\quad \sum_{k=N}^Ma_kx_k=\sum_{k=M+1}^La_kx_k.\]
Let $a^{(y,N,l)}=\sum_{k=N}^Ma_ke_k-\sum_{k=M+1}^La_ke_k$.
We then have $q_l(S_Ma^{(y,N,l)}-y)<\frac{1}{l}$ and $S_ka^{(y,N,l)}=0$ for any $k\ge L$.
We deduce that the conditions of Corollary \ref{cor X0} are satisfied for $X_0=\{a^{(y,N,l)}:y\in X,\ N\ge 0,\ l\ge 1 \}$.
\end{proof}

\begin{theorem}\label{ser phc}
Suppose that for any $n\ge 1$, $\ker p_n$ is a subspace of finite codimension.
If there exist $N\ge 0$ and $j\ge 1$ such that
\[\bigcup_{M\ge N}\left( \text{\upshape{span}}\{x_k:N\le k\le M \}\cap \Big(\text{\upshape{span}}\{x_k:M+1\le k\}+\ker q_j\Big)\right)\] is not dense in $X$,
then the sequence $(S_k)$ does not possess any hypercyclic subspace.
\end{theorem}
\begin{proof}
Let $N\ge 0$, $j\ge 1$ such that
\[\bigcup_{M\ge N}\left( \text{span}\{x_k:N\le k\le M \}\cap \Big(\text{span}\{x_k:M+1\le k\}+\ker q_j\Big)\right)\] is not dense in $X$.
Let $X_{j}=\{a\in A: \#\{k:q_j(S_{k}a)=0\}=\infty\}$.
We remark that 
\begin{multline*}
\bigcup_{M\ge N}S_M\big(\overline{\text{span}}\{e_k:k\ge N\}\cap X_j\big)\\
=\bigcup_{M\ge N}\left( \text{span}\{x_k:N\le k\le M \}\cap \Big(\text{span}\{x_k:M+1\le k\}+\ker q_j\Big)\right).
\end{multline*}
We conclude by Theorem \ref{thm E2}.
\end{proof}

In the case where $X$ possesses a continuous norm, we obtain the following generalization of the characterization given by Charpentier, Mouze and the author for the Fréchet space $\omega$ in \cite{Charpentier2}.

\begin{cor}
Suppose that $X$ possesses a continuous norm and for any $n\ge 1$, $\ker p_n$ is a subspace of finite codimension.
Then the sequence $(S_k)$ possesses a hypercyclic subspace if and only if 
for any $N\ge 0$,
\[\bigcup_{M\ge N}\left(\text{\upshape{span}}\{x_k:N\le k\le M \}\cap \text{\upshape{span}}\{x_k:M+1\le k\}\right)\] is dense in $X$.
\end{cor}

With a suitable adaptation of previous ideas, we can modify the condition of Theorem~\ref{condpas} to know when $\mathcal{U}_A$ is spaceable. We recall that we denote by ${S^A_k:A\rightarrow A}$ the operator defined by \[S^A_k((a_n)_{n\ge 0})=\sum_{n=0}^{k}a_ne_n.\]

\begin{theorem}
Suppose that for any $n\ge 1$, $\ker p_n$ is a subspace of finite codimension.
If for any $N\ge 0$, any $\varepsilon>0$,
\begin{multline*}
\bigcup_{L>M\ge N}\left( \Big\{\sum_{k=N}^Ma_kx_k:p_N\Big(\sum_{k=N}^Ma_ke_k\Big)<\varepsilon\Big\}\right.\\
\left.\cap\ \Big\{\sum_{k=M+1}^La_kx_k:p_N\Big(\sum_{k=M+1}^La_ke_k\Big)<\varepsilon\Big\}\right)
\end{multline*}
is dense in $X$, then $\mathcal{U}_A$ is spaceable.
\end{theorem}
\begin{proof}
Let $U$ be a non-empty open set in $X$, $N\ge 0$ and $\varepsilon>0$. By hypothesis, there exist a polynomial $a=\sum_{k=N}^La_ke_k\in A$ and $N\le M<L$ such that \[p_N(a)< \varepsilon,\ p_N(a-S^A_Ma)<\varepsilon,\ S_Ma\in U\ \text{and}\ S_La=0.\] 
Since $\ker p_N$ is a subspace of finite codimension, there even exists such a polynomial $a$ such that $p_N(a)=0$.
Indeed, if we let $d$ be the codimension of $\ker p_N$, we know that there exist polynomials $s_1,\dots,s_{d+1}\in A$ such that for any $1\le j\le d+1$, for some $v(s_j)\le M_j<d(s_j)$,
\[p_{N}(s_j)< \varepsilon,\ p_{N}(s_j-S^A_{M_j}s_j)<\varepsilon,\ S_{M_j}s_j\in U\ \text{and}\ S_{d(s_j)}s_j=0\]
and such that additionally $d(s_j)<v(s_{j+1})$ for any $1\le j\le d$. There then exist $\lambda_1,\dots,\lambda_{d+1}\in \K$  such that $\lambda_1 s_1+\cdots+\lambda_{d+1} s_{d+1}\in \ker p_N\backslash\{0\}$.
Moreover, since we can suppose that $\max |\lambda_i| \le 1$ and $\lambda_k=1$ for some $1\le k\le d+1$, we deduce that for any $N\ge 1$, any $\varepsilon>0$ and for any non-empty open set $U$, we can construct a non-zero polynomial $a$ with valuation as large as desired such that for some $v(a)\le M< d(a)$,
\[p_N(a)=0,\ p_N(a-S^A_Ma)<\varepsilon,\ S_{M}a\in U\ \text{and}\ S_{d(a)}a=0.\]

Let $(y_l)_{l\ge1}$ be a dense sequence in $X$. We construct recursively a family of polynomials $(a^{k,l})_{k,l\ge 1}$, for the strict order $\prec$ defined by $(k',l')\prec(k,l)$ if $k'+l'<k+l$ or if $k'+l'=k+l$ and $l'<l$, such that
\begin{itemize}
\item for any $(k',l')\prec (k,l)$, $d(a^{k',l'})<v(a^{k,l})$;
\item for any $k,l\ge 1$, there exists $v(a^{k,l})\le n_{k,l}< d(a^{k,l})$ such that 
\[ p_{k+l}(a^{k,l})=0,\ p_{l}(a^{k,l}-S^A_{n_{k,l}}a^{k,l})<\frac{1}{l}, q_l(S_{n_{k,l}}a^{k,l}-y_l)<\frac{1}{l}\ \text{and}\ S_{d(a^{k,l})}a^{k,l}=0  ;\]
\item for any $k\ge 1$, $p_{k+2}(a^{k,1})\ne 0$ (changing $p_{k+2}$ if necessary).
\end{itemize}
We can change $p_{k+2}$ during the construction because we consider this one only after the choice of $a^{k,1}$.

Let $a^{(k)}=\sum_{l=1}^{\infty} a^{k,l}$. We remark that we have 
$p_{k+1}(a^{(k)})=0$ and $p_{k+2}(a^{(k)})\ne 0$, and thus by Lemma \ref{bas ker} that \[\overline{\text{span}}\{a^{(k)}:k\ge 1\}=\Big\{\sum_{k=1}^{\infty}\alpha_ka^{(k)}:(\alpha_k)\in\mathbb{K}^{\mathbb{N}}\Big\}.\] 
Therefore, if $a=\sum_{k=1}^{\infty}\alpha_k a^{(k)}$ with $\alpha_{k_0}=1$, we deduce that for any $l_0\ge1$,
\begin{equation*}
q_{l_0}(S_{n_{k_0,l_0}}a-y_{l_0})
= q_{l_0}(S_{n_{k_0,l_0}}a^{k_0,l_0}-y_{l_0})<\frac{1}{l_0}\xrightarrow[l_0\rightarrow \infty]{} 0
\end{equation*}
and since $(k,l)\succ (k_0,l_0)$ implies $k+l\ge l_0$,
\begin{align*}
p_{l_0}(a-S^A_{n_{k_0,l_0}}a)&=p_{l_0}\Big(a^{k_0,l_0}-S^A_{n_{k_0,l_0}}a^{k_0,l_0}+\sum_{(k,l)\succ(k_0,l_0)}\alpha_k a^{k,l}\Big)\\
&=p_{l_0}(a^{k_0,l_0}-S^A_{n_{k_0,l_0}}a^{k_0,l_0})<\frac{1}{l_0}\xrightarrow[l_0\rightarrow \infty]{} 0.
\end{align*}
We conclude that $\overline{\text{span}}\{a^{(k)}:k\ge 1\}\subset \mathcal{U}_A\cup\{0\}$ and thus $\mathcal{U}_A$ is spaceable.

\end{proof}

\begin{theorem}\label{ser phc2}
Suppose that for any $n\ge 1$, $\ker p_n$ is a subspace of finite codimension.
If there exist $N\ge 0$, $j\ge 1$ and $\varepsilon>0$ such that
\begin{multline*}
\bigcup_{L>M\ge N}\left(\Big\{\sum_{k=N}^Ma_kx_k:p_N\Big(\sum_{k=N}^Ma_ke_k\Big)<\varepsilon\Big\}\right.\\
\left.\cap\ \Big(\Big\{\sum_{k=M+1}^La_kx_k:p_N\Big(\sum_{k=M+1}^La_ke_k\Big)<\varepsilon\Big\}+\ker q_j\Big)\right)
\end{multline*} is not dense in $X$,
then $\mathcal{U}_A$ is not spaceable.
\end{theorem}
\begin{proof}
Let $N\ge 0$, $j\ge 1$, $\varepsilon>0$ and $U\subset X$ a non-empty open set such that $U$ does not intersect
\begin{multline*}
\bigcup_{L>M\ge N}\left( \Big\{\sum_{k=N}^Ma_kx_k:p_N\Big(\sum_{k=N}^Ma_ke_k\Big)<\varepsilon\Big\}\right.\\ 
\left. \cap\Big(\Big\{\sum_{k=M+1}^La_kx_k:p_N\Big(\sum_{k=M+1}^La_ke_k\Big)<\varepsilon\Big\}+\ker q_j\Big)\right).
\end{multline*}
That means that for every sequence $a\in A$ with $v(a)\ge N$, if there exist ${v(a)\le M<L}$ such that $p_N(S^A_Ma)< \varepsilon$, $p_N(S^A_{L} a-S^A_M a)<\varepsilon$ and $S_{M} a\in U$, then $q_j(S_{L} a)\ne 0$.

Suppose that $\mathcal{U}_A\cup\{0\}$ contains a closed infinite-dimensional subspace $M_h$.
For any $n\ge 1$, since $\ker p_{N+n}$ is a subspace of finite codimension, there exits a non-zero sequence $a^{(n)}\in M_h\cap \ker p_{N+n}$ with valuation $v(a^{(n)})\ge N$. Since $a^{(n)}\in \mathcal{U}_A$, there also exists an integer $K_n\ge N+n$ such that $K_n\ge v(a^{(n)})$ and 
\[S_{K_n}a^{(n)}\in U\ \text{and}\ p_N(S^A_{K_n} a^{(n)})<\frac{\varepsilon}{2}.\] 
Thanks to the properties of $K_n$, we get that for any $n\ge 1$, any $k>K_n$, if $p_N(S^A_ka^{(n)})\le \frac{\varepsilon}{2}$, then $p_N(S_k^Aa^{(n)}-S^A_{K_n}a^{(n)})<\varepsilon$ and thus by properties of $N,j,\varepsilon,U$, we have $q_j(S_ka^{(n)})\ne 0$.
We deduce that for any $k>K_n$, we have  
\[q_j(S_ka^{(n)})\ne 0\ \text{ or }\  p_N(S^A_k a^{(n)})>\frac{\varepsilon}{2}.\]
We seek to construct $a=\sum_{n=1}^{\infty} \alpha_n a^{(n)}\in M_h$ such that for any $k> K_1$, we have $q_j(S_k a)\ge 1$ or $p_N(a-S^A_k a)\ge \frac{\varepsilon}{2}$. That will contradict the fact that each non-zero vector in $M_h$ is a restricted universal series.
We remark that since for any $n\ge 1$, $a^{(n)}\in \ker p_{N+n}$, the series $\sum_{n=1}^{\infty} \alpha_n a^{(n)}$ converges in $X$ for every sequence $(\alpha_n)\in \K^{\N}$, and since $M_h$ is a closed subspace,  $\sum_{n=1}^{\infty} \alpha_n a^{(n)}\in M_h$. Moreover, if $a=\sum_{n=1}^{\infty} \alpha_n a^{(n)}$, we have $p_N(a-S^A_k a)=p_N(S^A_k a)$. It thus suffices to construct a sequence $(\alpha_n)\in \K^{\N}$ such that for any $k> K_1$, we have 
\[q_j(S_k a)\ge 1\ \text{or}\ p_N(S^A_k a)\ge \frac{\varepsilon}{2}.\]
We thus begin by choosing $\alpha_1$ such that for any $K_1<k\le K_2$, we have
\[q_j(\alpha_1S_ka^{(1)})>1\ \text{or}\ p_N(\alpha_1 S^A_k a^{(1)})>\frac{\varepsilon}{2}.\]
Then, we choose $\alpha_n$ recursively such that for any $K_1<k\le K_{n+1}$, we have
\[q_j\Big(S_k\Big(\sum_{k=1}^{n-1}\alpha_k a^{(k)}\Big)+\alpha_n S_ka^{(n)}\Big)>1
\ \text{or}\ p_N\Big(S^A_k\Big(\sum_{k=1}^{n-1}\alpha_k a^{(k)}\Big) +\alpha_n S^A_ka^{(n)}\Big)>\frac{\varepsilon}{2}.\]
Such a choice exists because for any $K_n< k\le K_{n+1}$, we have either $q_j(S_k(a^{(n)}))\ne 0$ or $p_N(S^A_k a^{(n)})>\frac{\varepsilon}{2}$ and, for any $K_{1}<k\le K_{n}$, if we have $q_j(S_ka^{(n)})\ne 0$ or $p_N(S^A_k a^{(n)})\ne 0$, then it suffices to choose $\alpha_n$ sufficiently large and if we have $q_j(S_ka^{(n)})= 0$ and $p_N(S^A_k a^{(n)})= 0$, then we have the desired result by the induction hypothesis.
\end{proof}

\begin{cor}\label{spaceable}
Suppose that $X$ possesses a continuous norm and for any $n\ge 1$, $\ker p_n$ is a subspace of finite codimension.
Then the space $\mathcal{U}_A$ is spaceable if and only if for any $N\ge 0$, any $\varepsilon>0$,
\begin{multline*}
\bigcup_{L>M\ge N}\left( \Big\{\sum_{k=N}^Ma_kx_k:p_N\Big(\sum_{k=N}^Ma_ke_k\Big)<\varepsilon\Big\}\right.\\
\left.\cap\ \Big\{\sum_{k=M+1}^La_kx_k:p_N\Big(\sum_{k=M+1}^La_ke_k\Big)<\varepsilon\Big\}\right)
\end{multline*}
is dense in $X$.
\end{cor}

This Corollary together with Theorem \ref{univ} characterizes almost completely the spaceability of $\mathcal{U}_A$.

\subsection{Sequence of operators from $\omega$ to $\omega$}\label{KN}

In 2006, B\`{e}s and Conejero \cite{Bes} showed that for any non-constant polynomial $P$, for any weighted shift $B_w$, the operator $P(B_w):\omega\rightarrow \omega$ possesses a hypercyclic subspace. We are interested to know which sequences of operators from $\omega$ to $\omega$ possess a hypercyclic subspace. 
In the following, we denote by $(e_n)_{n\geq 0}$ the canonical basis of $\omega$ and by $p_n$ the seminorms defined by $p_n(x)=\max\{|x_k|:0\le k\le n\}$.

 Let $(T_k)$ be a sequence of continuous linear operators from $\omega$ to $\omega$.
These operators can be seen as matrices $(a^{(k)}_{i,j})_{i,j\ge 0}$, where $T_k e_j=(a^{(k)}_{i,j})_{i\ge 0}$, i.e. the $j$th column of $(a^{(k)}_{i,j})$ is given by $T_k e_j$. We notice that each row of the matrix $(a^{(k)}_{i,j})$ needs to be ultimately zero in order that $T_k$ is well-defined. We then denote by $c^{(k)}_i$ the smallest index such that for any $j\ge c^{(k)}_i$, $a^{(k)}_{i,j}= 0$. In particular, it means that for any $i\ge 0$, any $x\in \omega$, the $i$th coordinate of $T_kx$ only depends on the first $c^{(k)}_i$ coordinates of $x$.

By convention, if $A\subset \omega$, we let $A\cap\K^l:=\{(x_k)_{0\le k< l}:(x_k)_{k\ge 0}\in A\}$. In particular, we have 
\begin{equation}
T_k(\ker p_N)\cap \K^l=\text{\upshape{span}}\{(a^{(k)}_{i,j})_{i=0,\dots,l-1}:j\ge N+1\}.
\label{eq: al}
\end{equation}

We start our study of the existence of hypercyclic subspaces for $(T_k)$ by the following simple result:
\begin{lemma}\label{lemw}
If $(T_k)$ possesses a hypercyclic subspace, then for any $N\ge 0$, any $l\ge 1$, any $K\ge 1$,
\[\bigcup_{k\ge K} \text{\upshape{span}}\{(a^{(k)}_{i,j})_{i=0,\dots,l-1}:j\ge N\} \text{\quad is dense in $\K^l$}.\]
\end{lemma} 
\begin{proof}
If we suppose that $(T_k)$ possesses a hypercyclic subspace $M_h$, then we know that for any $N\ge 0$, we have $M_h\cap \ker p_N\ne \{0\}$ and thus, for any $K\ge 1$,
$\bigcup_{k\ge K}T_k(\ker p_N)$ is dense in $\omega$.
In particular, we deduce that for any $l\ge 1$, $\bigcup_{k\ge K}T_k(\ker p_N)\cap \K^l$ is dense in $\K^l$.
Thanks to \eqref{eq: al}, we therefore conclude that \[\bigcup_{k\ge K} \text{span}\{(a^{(k)}_{i,j})_{i=0,\dots,l-1}:j\ge N+1\}\ \text{is dense}.\]
\end{proof}

A condition a little stronger than the condition of Lemma \ref{lemw} gives us a sufficient condition for having a hypercyclic subspace.

\begin{theorem}\label{thm kn}
If for any $N\ge 0$, any $l\ge 1$, any $K\ge 1$, there exists $k\ge K$ such that
\[\text{\upshape{span}}\{(a^{(k)}_{i,j})_{i=0,\dots,l-1}:j\ge N\}=\K^l,\]
then $(T_k)$ possesses a hypercyclic subspace.
\end{theorem}
\begin{proof}
Thanks to Corollary \ref{cor X0}, we can conclude if we show that there exist a sequence $(n_k)$ and a set $X_0\subset X$ such that
\begin{enumerate}
\item for any $j\ge 1$, any $x\in X_0$, $(p_j(T_{n_k}x))_k$ is ultimately zero;
\item for any $K\ge 1$, any $l\ge 1$, any $y_0,\dots,y_{l-1}\in\K$, there exist $k\ge K$ and $x\in X_0$ such that $T_{n_k}x=(y_0,\dots, y_{l-1},*,*,\dots)$.
\end{enumerate}
By assumption, for any $N\ge 0$, any  $l\ge 1$, any $K\ge 1$, there exists $k\ge K$ and $M\ge N$ such that
\[\text{span}\{(a^{(k)}_{i,j})_{i=0\dots,l-1}:M>j\ge N\}=\K^l.\]
We can thus construct recursively two increasing sequences $(n_l)$ and $(m_l)$ with $m_1=1$ such that for any $l\ge 1$,
\begin{equation}
  \left\{
    \begin{split}
		&\text{span}\{(a^{(n_{l})}_{i,j})_{i=0,\dots,l-1}:m_{l+1}> j\ge m_l\}=\K^l\\ 
		&m_{l+1}\ge \max\{c^{(n_{l})}_{0},\dots,c^{(n_{l})}_{l-1}\}.
    \end{split}
  \right.
\label{eqt}
\end{equation}
The second inequality implies that for any $l\ge 1$, any $x\in \omega$, the first $l$ coordinates of $T_{n_l}x$ only depend on the first $m_{l+1}$ coordinates of $x$ and the first equality therefore implies that for any $l\ge 1$, any $y_0,\dots,y_{l-1}\in\K$, any $x_0,\dots, x_{m_l-1}\in\K$, there exists $x_{m_l},\dots,x_{m_{l+1}-1}\in \K$ such that
\[ T_{n_l}\big((x_0,\dots,x_{m_l-1},x_{m_l},\dots,x_{m_{l+1}-1},*,*,\dots)\big)=(y_0,\dots,y_{l-1},*,*,\dots).\]
In particular, for any $l\ge 1$, we can construct a sequence $x$ block by block thanks to \eqref{eqt} such that $T_{n_l}x=(y_0,\dots,y_{l-1},*,*,\dots)$ and for any $j> l$, $p_j(T_{n_j}x)=0$. The assumptions of Corollary \ref{cor X0} are thus satisfied for the sequence $(n_l)$.
\end{proof}
\begin{remark}
The condition
\[\text{\upshape{span}}\{(a^{(k)}_{i,j})_{i=0\dots,l-1}:j\ge N\}=\K^l\]
is equivalent to the existence of a number $M\ge N$ such that the column rank of the matrix $(a^{(k)}_{i,j})_{0\le i<l; N\le j\le M}$ is equal to $l$.
\label{rem rank}
\end{remark}

\begin{remark}
In general, the converse of Theorem \ref{thm kn} is false. If we consider a dense sequence $(y_l)_{l\ge 0}$ in $\omega$ and two applications $\varphi_1,\varphi_2:\N\rightarrow \Z_+$ such that 
\[\text{for any $n,l\ge 0$,}\quad \#\{k\in \N:(\varphi_1(k),\varphi_2(k))=(n,l)\}=\infty,\]
then the sequence $(T_k)_{k\ge 1}$, defined by $T_k(e_{n})=y_{\varphi_2(k)}\delta_{n,\varphi_1(k)}$, possesses $\omega$ as hypercyclic subspace but the assumption of Theorem \ref{thm kn} is not satisfied. However it seems more difficult to exhibit a counterexample in the case of an operator $T$ and its iterates. Therefore one may wonder if, in this case, the assumption of Theorem~\ref{thm kn} is equivalent to having a hypercyclic subspace.
\end{remark}

\begin{samepage}
An immediate corollary of Theorem \ref{thm kn} can be stated as follows:
\begin{cor}\label{cor c}
If there exist two increasing sequences of integers $(i_k)$ and $(n_k)$ such that
\begin{enumerate}[\upshape (i)]
\item for any $i\ge 0$, $c^{(n_k)}_i\rightarrow \infty$ as $k\rightarrow\infty$,
\item for any $k\ge 1$, any $i,j\le i_k$, $i\ne j\Rightarrow c^{(n_k)}_i\ne c^{(n_k)}_j$,
 \end{enumerate}
then $(T_k)$ possesses a hypercyclic subspace. 
\end{cor}
\begin{proof}
Let $j^{k,l}_0=\min_{0\le i <l}c^{(n_k)}_i-1$ and $j^{k,l}_1=\max_{0\le i <l}c^{(n_k)}_i$. 
We deduce from (ii) that for any $k\ge 1$, any $l\le i_k$, the column rank of the matrix $(a^{(n_k)}_{i,j})_{0\le i <l;\ j^{k,l}_0\le j\le j^{k,l}_1}$ is equal to $l$. By (i), we also have that for any $l\ge 1$, $j^{k,l}_0\rightarrow \infty$ as $k\rightarrow \infty$. We conclude thanks to Theorem~\ref{thm kn} and Remark~\ref{rem rank}.
\end{proof}
\end{samepage}

We remark that the conditions of Corollary~\ref{cor c} do not depend on the values of the matrices $(a^{(k)}_{i,j})$ but only on the positions of non-zero elements of these matrices. In the case of an operator $T$ and its iterates, if we denote by $c_i$ the smallest index such that for any $j\ge c_i$, $a_{i,j}=0$, where $T=(a_{i,j})_{i,j\ge 0}$, then we can obtain a necessary condition only depending on values $c_i$.

\begin{cor}\label{cor bes}
Let $T:\omega\rightarrow \omega$ be a continuous linear operator. If 
\begin{enumerate}[\upshape (i)]
\item for any $i\ge 0$, $c_i>i+1$,
\item  for any $i\ne j$, $c_i\ne c_j$,
 \end{enumerate}
 then $T$ possesses a hypercyclic subspace. 
\end{cor}
\begin{proof}
Let $n\ge 1$. By definition of $c_i$, we have \[T(\ker p_n)\subset\{x\in \omega:x_i=0 \ \text{ if } c_i\le n+1\}\] and by (ii), we also have \[T(\ker p_n)\supset\{x\in \omega:x_i=0 \ \text{ if } c_i\le n+1\}.\] Therefore, by (i), we get
\begin{align*}
T(\ker p_n)&=\{x\in \omega:x_i=0 \ \text{ if } c_i\le n+1\}\\
&\supset \{x\in \omega:x_i=0 \ \text{ if } i< n\}=\ker p_{n-1}.
\end{align*}
We deduce that $T^n(\ker p_n)\supset\ker p_0$ and since
\[T(\ker p_0)=\{x\in \omega:x_i=0 \ \text{ if } c_i\le 1\}=\omega,\]
we conclude thanks to \eqref{eq: al} and Theorem \ref{thm kn}.
\end{proof}

Corollary \ref{cor bes} gives us a large class of operators with hypercyclic subspaces on $\omega$. In particular, we remark that the operators $P(B_w)+\sum_{k=1}^{\infty} \alpha_k S_w^k$, where $P$ is a non-constant polynomial and $S_w$ is a weighted forward shift, possess a hypercyclic subspace. This contains, in particular, the result of B\`{e}s and Conejero \cite{Bes}.\\

We end this section by looking at the notion of frequently hypercyclic subspaces for operators from $\omega$ to $\omega$. Let $(T_k)$ be a sequence of linear continuous operators from $X$ to $Y$, where $X$ is an infinite-dimensional Fréchet space and $Y$ is a separable topological vector space. The sequence $(T_k)$ is said to be frequently hypercyclic if there exists a vector $x$ in $X$ (also called frequently hypercyclic) such that for any non-empty open set $U\subset X$, the set $\{k\ge 1: T_kx\in U\}$ is of positive lower density. We therefore say that $(T_k)$ possesses a frequently hypercyclic subspace if there exists a closed infinite-dimensional subspace in which every non-zero vector is frequently hypercyclic.  The notion of frequently hypercyclic operator has been introduced by Bayart and Grivaux \cite{Bayart0}, and the notion of frequently hypercyclic subspace has been studied for the first time by Bonilla and Grosse-Erdmann \cite{Bonilla2}.

\begin{theorem}\label{fhc}
Let $f_{k,l}=\max\{c^{(k)}_i:i=0,\dots,l-1\}$.
If there exist $(d_{k})_{k\ge 1}$ and $(N_l)_{l\ge 1}$ such that for any $k,l\ge 1$, we have
\begin{equation}\text{\upshape{span}}\{(a^{(k)}_{i,j})_{i=0\dots,l-1}:d_{k}\le j< f_{k,l}\}=\K^l\label{df}\end{equation}
and for any $k'\ge k+N_l$, we have $f_{k,l}\le d_{k'}$, then $(T_k)$ possesses a frequently hypercyclic subspace.
\end{theorem}
\begin{proof}
Let $(y_l)_{l\ge 1}$ be a dense sequence in $\omega$. Let $(f_{k,l})_{k,l\ge 1}$, $(d_{k})_{k\ge 1}$, and $(N_l)_{l\ge 1}$ be sequences as given in the statement. We know (see \cite{Bayart}, \cite{Bonilla}) that there exists a family of disjoint sets $A(l,\nu)\subset \N$ of positive lower density such that for any $n\in A(l,\nu)$, for any $m\in A(\lambda,\mu)$, we have
$n\ge \nu$ and $|n-m|\ge \nu +\mu$ if $n\ne m$.
We let $E_{l,j}=A(l,j+N_l)$. The goal of this proof is to construct a sequence $(u_j)_{j\ge 1}\subset\omega$ such that $v(u_j)\rightarrow \infty$  and for any $(\alpha_j)_j\in \K^{\N}$, the sequence $\sum_{j=1}^{\infty}\alpha_j u_j$ is frequently hypercyclic. To this end, we will construct the sequence $(u_j)_j$ such that for any $l,j\ge 1$, for any $n\in E_{l,j}$, for any $j'\ne j$, we have
\[ p_{l-1}(T_nu_j-y_l)=0 \quad\text{and}\quad p_{l-1}(T_nu_{j'})=0.\]
We will construct these sequences block by block by going through each element in ${E:=\bigcup_{l,j\ge 1} E_{l,j}}$. We denote by $n_k$ the $k$th element in $E$ and  \[J_k=\Big\{j\ge 1: \{n_1,\dots,n_{k-1}\}\cap\bigcup_{l\ge 1}E_{l,j}\ne\emptyset\Big\}.\]
We remark that for any $k\ge 1$, if $n_k\in E_{l,j}$ and $n_{k+1}\in E_{l',j'}$, then $n_{k+1}-n_k\ge j+ N_l + j' +N_{l'}\ge N_l$ and thus by hypothesis, we have $f_{n_k,l}\le d_{n_{k+1}}$. Our construction will therefore consist at the $k$th step in adding to some sequences $u_i$ a new block $u_{i,d_{n_k}},\dots, u_{i,f_{n_k,l}-1}$ if $n_k\in E_{l,j}$. In this way, if $n_k\in E_{l,j}$, the first $l$ coordinates of $T_{n_k}u_i$ will not be modified after the $k$th step of our construction by definition of $f_{n_k,l}$. The sets $J_k$ will permit us to know for which sequence $u_j$ some blocks have already been constructed after the $(k-1)$st step. Our construction is the following:
\begin{itemize}
\item If $n_k\in E_{l,j}$ and $j\notin J_k$, we let $u_j=(0,\dots,0,u_{j,d_{n_k}},\dots, u_{j,f_{n_k,l}-1},*,*,\dots)$ such that
$p_{l-1}(T_{n_k}u_j-y_l)=0$ and for each $i\in J_k$, we complete $u_i$ by some block $u_{i,d_{n_k}},\dots u_{i,f_{n_k,l}-1}$ such that $p_{l-1}(T_{n_k}u_i)=0$. Such blocks exist by \eqref{df}.
\item If $n_k\in E_{l,j}$ and $j\in J_k$, we complete $u_j$ by some block $u_{j,d_{n_k}},\dots u_{j,f_{n_k,l}-1}$ such that $p_{l-1}(T_{n_k}u_j-y_l)=0$ and 
for each $i\in J_k$, $i\ne j$, we complete $u_i$ by some block $u_{i,d_{n_k}},\dots u_{i,f_{n_k,l}-1}$ such that $p_{l-1}(T_{n_k}u_i)=0$.
\end{itemize}
The condition \eqref{df} implies that $f_{k,l}\ge l$. Hence, $d_k\rightarrow \infty$ as $k\rightarrow \infty$ and thus $v(u_j)\rightarrow \infty$ as $j\rightarrow \infty$. We deduce that there exists a subsequence of $(u_j)$ and a subsequence of $(p_n)$ such that the condition of Lemma \ref{bas ker} is satisfied. The subspace obtained is therefore a frequently hypercyclic subspace for $(T_k)$.
\end{proof}

In the case of an operator $T$ and its iterates, we obtain the following simple criterion.
\begin{cor}
Let $T:\omega\rightarrow \omega$ be an operator and $r\ge 1$ an integer. If $c_0\ge 2$ and for any $i\ge 0$, $c_{i+1}=c_i+r$, then $T$ possesses a frequently hypercyclic subspace. 
\end{cor}
\begin{proof}
By induction, we can show that $c^{(k)}_i=\sum_{n=0}^{k-1}r^n (c_0-1)+r^ki+1$ because for any $k\ge 2$, we have \[c^{(k)}_i=c^{(k-1)}_{c_i-1}.\] 
We thus have 
\[f_{k,l}:=\max\{c^{(k)}_i:i=0,\dots,l-1\}=\sum_{n=0}^{k-1}r^n (c_0-1)+r^k(l-1)+1.\] Therefore, letting $d_{k}=c^{(k)}_0-1$, we get for any $l\ge 1$,
\[\text{span}\{(a^{(k)}_{i,n})_{i=0\dots,l-1}:d_{k}\le n< f_{k,l}\}=\K^l\]
and for any $l,k\ge 1$, for any $k'\ge k+l$, we have 
\begin{align*}
f_{k,l}&= \sum_{n=0}^{k-1}r^n (c_0-1)+r^k(l-1)+1\\
&\le \sum_{n=0}^{k+l-1}r^n (c_0-1) = c^{(k+l)}_0-1= d_{k+l}\le d_{k'}.
\end{align*}
We conclude by Theorem \ref{fhc}.
\end{proof}

\begin{cor}
Each operator $P(B_w)+\sum_{k=1}^{\infty} \alpha_k S_{w}^k$, where $P$ is a non-constant polynomial and $S_w$ is a weighted forward shift, possesses a frequently hypercyclic subspace on $\omega$.
\end{cor}

\subsection{Weighted shifts}\label{weighted}

Weighted shifts are classical examples of hypercyclic operators. In \cite{Menet2}, a characterization of weighted shifts with hypercyclic subspaces on certain Köthe sequence spaces with a continuous norm is given. In this subsection, we complement this result by giving a characterization of weighted shifts with hypercyclic subspace of type $2$ on certain Fréchet sequence spaces.
 
A Fréchet space $X$ is a \emph{Fréchet sequence space}
if $X$ is a subspace of $\omega$ such that the embedding of $X$ in $\omega$ is continuous. In other words, $X$ is a Fréchet sequence space if the convergence in $X$ implies the convergence coordinates by coordinates. A family of Fréchet sequence spaces is given by Köthe sequence spaces. Let $A=(a_{j,k})_{j\ge 1,k\ge 0}$ be a matrix such that for any $j\ge 1$, any $k\ge 0$, we have  $a_{j,k}\le a_{j+1,k}$ and for any $k\ge 0$, there exists $j\ge 1$ such that $a_{j,k}\ne 0$.
We define the (real or complex) \emph{Köthe sequence spaces} $\lambda^p(A)$ with $1\le p<\infty$ and $c_0(A)$ by
\begin{align*}
\lambda^p(A)&:=\Big\{(x_k)_{k\ge 0}\in \omega : p_j((x_k)_k)=\Big(\sum_{k=0}^{\infty}|x_ka_{j,k}|^p\Big)^{\frac 1 p}<\infty, \ j\ge 1\Big\},\\
c_0(A)&:=\{(x_k)_{k\ge 0}\in \omega: \lim_{k\rightarrow \infty}|x_k|a_{j,k}=0, \ j\ge 1\} \text{ with } p_j((x_k)_k)= \sup_{k\ge 0}|x_k|a_{j,k}.
\end{align*}
These Fréchet spaces possess a continuous norm if and only if there exists $j\ge 1$ such that for any $k\ge 0$, $a_{j,k}\ne 0$.

We fix a Fréchet sequence space $X$ and a weighted shift $B_w:X\rightarrow X$ defined by $B_we_n=w_ne_{n-1}$, where $e_{-1}=0$, $(e_n)_{n\ge 0}$ is the canonical basis in $\omega$ and $(w_n)_{n\ge 1}$ is a sequence of non-zero scalars.

\begin{theorem}\label{thm ws}
Suppose that the set of finite sequences is a dense subset of $X$. If for any continuous seminorm $p$ on $X$, the set $\{k\ge 0:p(e_k)=0\}$  contains arbitrarily long intervals, then $B_w$ possesses a hypercyclic subspace of type $2$.
\end{theorem}
\begin{proof}
Let $X_0$ be the set of finite sequences. We can prove that the assumptions of Theorem \ref{thm X0} are satisfied for $X_0$ and the whole sequence $(n)$.
Indeed, for any $x\in X_0$, we remark that the sequence $(B_w^nx)_n$ is ultimately zero. Moreover, for any continuous seminorm $p$ on $X$, for any $n\ge 1$, we know that there exists $k\ge 0$ such that
\[p(e_{k})=\cdots=p(e_{k+n})=0.\] Let $x_0,\dots,x_n\in \K$.
We have thus that
\[y:=\sum_{l=0}^{n}\frac{x_l}{\prod_{\nu=1}^{k}w_{l+\nu}}e_{k+l}\in \ker p\cap X_0
\text{\ and \ }B^k_wy=\sum_{l=0}^nx_le_l.\]
Since the finite sequences are dense in $X$, we conclude by Theorem \ref{thm X0}.
\end{proof}

If there exists an increasing sequence of seminorms $(p_j)_{j\ge 1}$ inducing the topology of $X$ such that for any $(x_k)_{k\ge 0}$, $(y_k)_{k\ge 0}\in X$, we have
\[(|x_k|\le |y_k| \text{ for any $k\ge 0$})\Rightarrow p_j((x_k)_{k\ge 0})\le p_j((y_k)_{k\ge 0}),\]
then it is interesting to notice that if the assumption of Theorem \ref{thm ws} is not satisfied, the existence of a weighted shift $B_w:X\rightarrow X$ implies that $X$ possesses a continuous norm. 

\begin{lemma}\label{lem cont}
Let $(p_j)_{j\ge 1}$ be an increasing sequence of seminorms inducing the topology of $X$. Suppose that the set of finite sequences is a subset of $X$ and that for any $(x_k)_{k\ge 0}$, $(y_k)_{k\ge 0}\in X$, we have
\begin{equation}
(|x_k|\le |y_k| \text{ for any $k\ge 0$})\Rightarrow p_j((x_k)_{k\ge 0})\le p_j((y_k)_{k\ge 0}).
\label{ideakarl}
\end{equation}
If there exist a weighted shift $B_w$ from $X$ to $X$ and a continuous seminorm $p$ on $X$ such that the set $\{k\ge 0:p(e_k)=0\}$ does not contain arbitrarily long intervals, then $X$ possesses a continuous norm.
\end{lemma} 
\begin{proof}
By hypothesis, there exist  a continuous seminorm $p$ on $X$ and $N\ge 1$ such that for any $k\ge 0$, we have $p(e_{k+i})\ne 0$ for some $0\le i\le N$. Without loss of generality, we can suppose that $p(e_0)\ne 0$. Moreover, since $B_w$ is a weighted shift from $X$ to $X$ and $X$ is a Fréchet sequence space, we know that $B_w$ is continuous. Therefore, there exist $j\ge 1$ and a constant $C>0$ such that for any $k\ge 0$, for any $0\le i\le N$, we have \[p(B^i_w e_k)\le C p_j(e_k).\]
In particular, for any $k\ge 0$, if $p(e_k)\ne 0$, then $p_j(e_{k+i})\ne 0$ for any $0\le i\le N$. Thanks to properties of $N$ and the fact that $p(e_0)\ne 0$, we deduce that for any $k\ge 0$, $p_j(e_k)\ne 0$.
Using \eqref{ideakarl}, we therefore conclude that $p_j$ is a norm.
\end{proof}

As hypercyclic subspaces of type $2$ do not exist on Fréchet spaces with a continuous norm, we obtain the following characterization:
\begin{theorem} 
Let $X$ be a Fréchet sequence space such that the set of finite sequences is a dense subset of $X$.
Let $B_w$ be a weighted shift from $X$ to $X$.
Let $(p_j)_{j\ge 1}$ be an increasing sequence of seminorms inducing the topology of $X$. Suppose that for any $(x_k)_{k\ge 0}$, $(y_k)_{k\ge 0}\in X$, we have
\begin{equation*}
(|x_k|\le |y_k| \text{ for any $k\ge 0$})\Rightarrow p_j((x_k)_{k\ge 0})\le p_j((y_k)_{k\ge 0}).
\end{equation*}
Then the operator $B_w$ possesses a hypercyclic subspace of type $2$ if and only if for any $j\ge 1$, the set $\{k\ge 0:p_j(e_k)=0\}$ contains arbitrarily long intervals.
\end{theorem}
\begin{cor}
Let $X=\lambda^p(A)$ or $c_0(A)$ and $B_w$ a weighted shift from $X$ to $X$.
The operator $B_w$ possesses a hypercyclic subspace of type $2$ if and only if for any $j\ge 1$, the set $\{k\ge 0:a_{j,k}=0\}$ contains arbitrarily long intervals.
\end{cor}

\section*{Acknowledgment}

The author would like to thank the referee for valuable comments.

\end{document}